\documentclass{article}

\usepackage{PRIMEarxiv}
\usepackage[numbers]{natbib}
\usepackage[utf8]{inputenc} 
\usepackage[T1]{fontenc}    
\usepackage{hyperref}       
\usepackage{url}            
\usepackage{booktabs}       
\usepackage{amsfonts}       
\usepackage{amsmath}
\usepackage{nicefrac}       
\usepackage{microtype}      
\usepackage{lipsum}
\usepackage{todonotes}
\usepackage{fancyhdr}       
\usepackage{graphicx}       
\usepackage{subcaption} 
\graphicspath{{media/}}     
\usepackage{placeins}
\usepackage{orcidlink}
\usepackage{algorithm}
\usepackage{algorithmic}
\usepackage[pagewise]{lineno}
\usepackage{mwe} 
\usepackage{multirow}
\usepackage{doi}
\usepackage{hyperref}

\usepackage{amsthm}

\newtheorem{proposition}{Proposition}

\title{Pruning for efficient deterministic global optimization over trained ReLU neural networks}

\author{
  Giacomo Lastrucci\orcidlink{0009-0005-3475-9351}\\
  Process Intelligence Research \\
  Department of Chemical Engineering \\
  Delft University of Technology \\
  Delft, The Netherlands\\
   \And
  Tanuj Karia\orcidlink{0009-0004-7720-2628} \\
  Process Intelligence Research \\
  Department of Chemical Engineering \\
  Delft University of Technology \\
  Delft, The Netherlands\\
   \And
  Victor Schulte\orcidlink{0009-0000-8475-4261} \\
  Process Intelligence Research \\
  Department of Chemical Engineering \\
  Delft University of Technology \\
  Delft, The Netherlands\\
   \And
  Dominik Bongartz\orcidlink{0000-0003-1790-0235} \\
  Department of Chemical Engineering \\
  KU Leuven \\
  Leuven, Belgium\\
   \And
  Artur M. Schweidtmann\orcidlink{0000-0001-8885-6847} \\
  Process Intelligence Research \\
  Department of Chemical Engineering \\
  Delft University of Technology \\
  Delft, The Netherlands\\
  \texttt{a.schweidtmann@tudelft.nl} \\
}

\begin{document}
\maketitle

\begin{abstract}
Neural networks are increasingly used as surrogates in optimization problems to replace computationally expensive models.
However, embedding ReLU neural networks in mathematical programs introduces significant computational challenges, particularly for large networks.
This work investigates how pruning techniques can accelerate the solution of optimization problems with embedded neural networks, focusing on the mechanisms underlying the computational gains.
We show that, besides problem size reduction, both unstructured and structured pruning tightens preactivation bounds.
We conduct a comprehensive empirical study spanning three morphologies and four complexity levels, comparing magnitude-based weight and node pruning with structured perspective regularization (SPR). We test on both an illustrative test function and a chemical process flowsheet surrogate, using a commercial and an open-source solver.
Pruning consistently improves the accuracy-runtime trade-off, with speedups spanning one to four orders of magnitude depending on the complexity of the unpruned network.
The gains are attributed to three factors: (i) reduction in problem size, (ii) decrease in the number of integer variables, and (iii) tightening of big-M bounds.
Weight pruning and SPR occupy most of the accuracy-runtime Pareto frontier, the latter enabling faster convergence at the mildest sparsity considered.
The mechanisms transfer from the test function to the engineering case study, although the preferred technique and sparsity level might vary with the case study, and should be selected as part of surrogate development.
\end{abstract}



\begin{figure}[h!]
    \centering
    \includegraphics[width=0.8\linewidth]{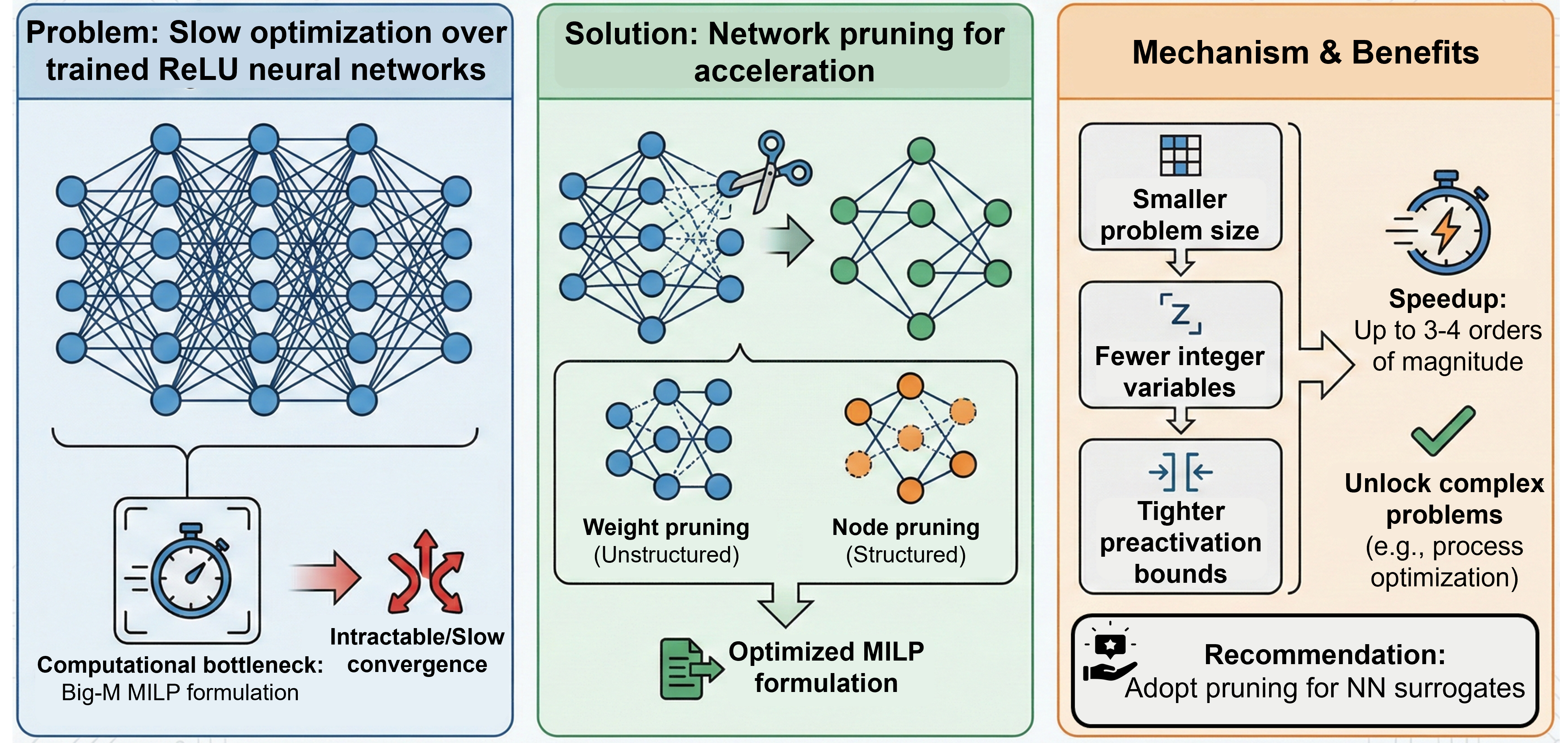}
    \caption{Optimization over embedded ReLU neural networks is often slow or intractable; pruning reduces formulation complexity and accelerates deterministic global optimization, at a moderate accuracy cost.}
    \label{fig:graphical_abstract}
\end{figure}

\section{Introduction}
Neural networks (NNs) have been widely used as a surrogate model for mathematical optimization tasks in multiple branches of engineering~\cite{Schweidtmann2018_DeterministicGlobalOptimization, White2019_Multiscaletopologyoptimization, Fischetti2019_Machinelearningmeets, Sun2019_reviewartificialneural, Lastrucci2025_PicardKKThPINN}.
Surrogate modeling addresses the challenge of tractability by replacing complex, computationally expensive, or unknown mathematical models with simpler-to-evaluate approximations. 
In the anatomy of optimization problems, surrogate models can replace objective functions or constraints. 
In an optimization problem, a set of equality constraints can represent, for instance, physical laws expressed as differential-algebraic equations. 
Replacing a set of constraints with a surrogate model is known in the literature as \textit{constraint learning}~\cite{Serra2017_BoundingCountingLinear, Donti2021_DC3learningmethod, Lastrucci2025_ENFORCENonlinearConstrained} and can improve the mathematical tractability, e.g., by avoiding domain discretization or decreasing the overall problem size.\\
Deterministic global optimization with NNs embedded requires expressing the neural network's internal computations as an algebraic system of (nonlinear) equations. Several software packages have been developed to automatically reformulate NNs for optimization, including JANOS~\cite{Bergman2019_JANOSIntegratedPredictive}, optiCL~\cite{Maragno2021_MixedIntegerOptimization}, reluMIP~\cite{reluMIP.2021}, OMLT~\cite{Ceccon2022_OMLTOptimization&}, and the MeLON interface to MAiNGO~\cite{Bongartz2018MAiNGO}, as well as solver-specific interfaces for GUROBI and SCIP~\cite{Turner2023_PySCIPOptMLEmbedding}. 
The nonlinearity of the neural network input–output mapping is determined by the choice of activation function. However, the nonlinearity of the resulting optimization formulation also depends on how the network is represented in the mathematical program.
While smooth, continuous activations such as GeLU are gaining popularity in the machine learning domain~\cite{Hendrycks2016_GaussianErrorLineara}, the rectified linear unit (ReLU) is still arguably one of the most popular choices for constraint learning due to two main factors: (i) efficient gradient propagation during training, yielding good approximation capabilities~\cite{nair2010rectified}, and (ii) piecewise linearity, which allows reformulating the network (being overall a piecewise linear function) as a mixed-integer linear program (MILP)~\cite{Huchette2023_WhenDeepLearning}. 
In fact, the most widely used approach for embedding ReLU networks in optimization problems is the big-M formulation, which provides a standard mechanism for modeling if–else statements in mathematical programming~\cite{Grimstad2019_ReLUnetworksasa}. The resulting MILP can then be solved to global optimality using deterministic solvers based on branch-and-bound (B\&B) algorithms~\cite{Grossmann2021_AdvancedOptimizationProcess}.
The drawback of this approach is the combinatorial nature of mixed-integer programs~(MIPs) whose complexity grows rapidly with the number of neurons in the network.\\ 
The efficiency of solving the resulting MILP strongly depends on the strength of its linear programming (LP) relaxation. 
As a result, extensive research has focused on strengthening big-M formulations for ReLU networks, ranging from reformulations~\cite{Tsay2021_PartitionBasedFormulations}, to variable bounds tightening~\cite{Sosnin2024_ScalingMixedInteger}, and cutting planes~\cite{Anderson2020_Strongmixedinteger}. 
Recent studies have also examined additional factors in ReLU NNs affecting MILP solver performance, including the number of linear regions and stable neurons~\cite{Serra2017_BoundingCountingLinear,Plate2025_analysisoptimizationproblems}. 
Despite these advances, optimization problems with embedded ReLU networks remain computationally challenging, especially for deep and large networks.
A potential technique to reduce the computational burden of constraint learning with ReLU surrogates is pruning.\\
Pruning is a machine learning technique that aims to reduce the size of a neural network by removing unnecessary connections (parameters) or entire structures (neurons, kernels). The origins of pruning date back to the late 80s, when the need to compensate for modest computational power and the goal of understanding NNs' learning behavior drove researchers to investigate the impact of individual parameters on the functional relationship, leading to its eventual simplification~\cite {Sietsma_1988, Mozer1988SkeletonizationAT, MOZER1989_UsingRelevanceReduce}.
After the "AI winter" period (1990s-2000s), pruning lived a renaissance, with dozens of studies proposing methods to reduce the computational and memory demands of modern deep NNs, especially for deployment on resource-constrained systems (e.g., smartphones and edge devices)~\cite{Han2015_LearningbothWeights, Blalock2020_WhatisState}.
Different pruning strategies have been proposed in the literature and can be broadly categorized along two dimensions: the entity being removed, distinguishing unstructured and structured pruning, and the criterion used for removal, such as magnitude-based or sensitivity-based pruning.\\
While pruning has been extensively studied for model compression and deployment, its role in improving the computational efficiency of optimization problems with embedded NNs has only recently started to receive attention. 
Cacciola et al.~\cite{Cacciola2023_StructuredPruningNeural} proposed a theoretically grounded structured pruning method, namely, Structured Perspective Regularization (SPR), for MIPs with embedded NNs. Their approach promotes structured sparsity, enabling the removal of entire network structures and associated integer variables from the optimization formulation, thereby substantially accelerating optimization. 
Pham et al.~\cite{Pham2025_OptimizationoverTrained} considered the complementary problem of employing a sparsified network as a tractable proxy for optimization over a given dense one, and reported that cost-effective sparsification, even without fine-tuning, performs well in network verification tasks. Plate et al.~\cite{Plate2025_analysisoptimizationproblems} investigated the
influence of training and post-training techniques on the number of linear regions, the fraction of stable neurons, and the resulting solver runtime.
These studies provide important evidence for the benefits of pruning in optimization. 
Three aspects remain to be characterized.
First, the relative contribution of the individual formulation properties to the observed speedups, and its dependence on the depth, width and size of the pruned network. Second, the performance of traditional magnitude-based criterion compared with criterions tailored to MIP embedding such as SPR. Third, the transferability of the resulting outcomes to engineering surrogate modeling, where the network approximates a process model.\\
The present work addresses these aspects through analytical and empirical
evidence, benchmarking more than 1,000 trained ReLU networks across morphologies, complexity levels, pruning techniques and MILP solvers. Our main findings are: (1) pruning accelerates the optimization by one to four orders of magnitude, with larger gains for more complex unpruned networks; (2) besides reducing problem size and binary variables, pruning monotonically tightens big-M bounds, an effect that compounds with depth so that deep and narrow networks gain the most; (3) a simple magnitude criterion suffices for small networks, whereas SPR restores tractability of the largest networks at the mildest sparsity; (4) these mechanisms transfer to a chemical process flowsheet surrogate and to an open-source solver, although the preferred technique and sparsity level depend on the case study.\\
The remainder of this paper is organized as follows. Section~\ref{sec:background} provides background on the magnitude-based pruning techniques used in this study and on the big-M formulation of ReLU networks, and positions the present work with respect to the closest related studies. Section~\ref{sec:methods} analyzes how pruning affects the algebraic formulation and solver performance, showing analytically that pruning leads to tighter bounds in the big-M formulation. 
These findings are then confirmed with a detailed empirical evaluation in Section~\ref{sec:experiments}, where the impact of pruning on optimization problems is assessed on a test function and an engineering case study.

\section{Background}
\label{sec:background}
This section provides the necessary background on neural-network pruning, including unstructured and structured methods, and on ReLU big-$M$ formulations and bound-tightening techniques. We conclude by positioning the present work with respect to the closest related studies.

\subsection{Pruning of neural networks}
\label{sec:pruning_nn}
Pruning aims to reduce the size of a neural network by cutting unnecessary connections or entire structures.
Almost all strategies to prune NNs can be viewed as a subclass of a common algorithmic pattern. Therein, a (large) neural network is firstly trained to convergence. Afterward, the network’s weights and biases are evaluated based on a specified criterion and assigned a score accordingly. Based on these scores, parameters are selected and removed by masking them in the weight matrices or bias vectors. This process can be iteratively repeated~\cite{Han2015_LearningbothWeights}. Optionally, the network can be fine-tuned to partially recover accuracy potentially lost during pruning by refining the remaining parameters~\cite{Renda2020_ComparingRewindingFine}.\\
One way to characterize pruning methods is to classify them in \textit{unstructured} and \textit{structured} methods. In unstructured pruning (also known as \textit{weight pruning}), individual weights are eliminated from the network to produce sparse counterparts. In structured pruning, entire groups of parameters are removed. For instance, in fully-connected NNs, structured pruning can be used to remove entire nodes by cutting both input and output weights of a neuron (\textit{node pruning}). In contrast to the random sparsity introduced by unstructured pruning, structured pruning reduces network size by eliminating entire structures, such as rows within weight tensors (nodes) or even entire tensors (layers).\\
Pruning methods also differ in the criterion used to determine which parameters to prune. \textit{Magnitude-based} pruning assumes that small weights have a negligible impact on the prediction of a neural network~\cite{Han2015_LearningbothWeights}. Consequently, weights are pruned based on their absolute values, either by applying a threshold below which they are removed or by specifying a target percentage of parameters to remove (i.e., sparsity). In the latter, parameters are progressively pruned, starting from the smallest values until the desired sparsity is reached. A more complex alternative is given by \textit{sensitivity-based} pruning, which finds the contribution of each parameter to the overall loss function using sensitivity analysis~\cite{LeCun1989_OptimalBrainDamage}. Parameters with low influence on the loss function are then pruned from the network. Sensitivity-based pruning requires computing or approximating the first and second derivatives of the loss function with respect to each parameter, leading to a significantly higher computational cost compared to the simpler class of magnitude-based methods~\cite{Pham2025_OptimizationoverTrained}.\\
Beyond model compression, pruning provides a useful framework for improving the efficiency of optimization problems with trained NNs embedded. In fact, the elimination of parameters and structures from the neural network tensors affects the network's algebraic formulation.

\subsection{Iterative magnitude-based pruning}
\label{sec:iterative_pruning}
We focus on magnitude-based weight and node pruning, following the iterative pruning approach of Han et al.~\cite{Han2015_LearningbothWeights}. 
Sensitivity-based pruning methods have shown strong performance~\cite{Yu2022_CombinatorialBrainSurgeon, Pham2025_OptimizationoverTrained}, but they typically require additional derivative- or optimization-based computations. Therefore, to keep the analysis focused, we restrict the main study to widely used magnitude-based pruning strategies.
We also include SPR (Section~\ref{sec:regularization}) as a structured regularization-based pruning approach tailored for optimization with embedded neural networks.

\subsubsection{Weight pruning}
In magnitude-based weight pruning, a target sparsity level $s_f$ is specified. The weights are then ranked according to their absolute values, and those with the smallest magnitude are pruned until the desired sparsity is reached. The target sparsity is defined as $s_f = \frac{n_p}{n_t}$, where $n_p$ is the number of weights pruned from the neural network and $n_t$ is the total number of weights of the unpruned network.\\
Similar to Zhu et al.~\cite{Zhu2017_prunenotprune}, we extend the original pruning method of Han et al.~\cite{Han2015_LearningbothWeights} by adopting an iterative pruning strategy with a constant relative pruning rate across iterations. Specifically, at each pruning iteration, the same fraction $s_r$ of the weights that remain unpruned is removed. This iterative procedure helps mitigate the loss in predictive accuracy that may occur when pruning is performed in a single step. The relative pruning rate $s_r$ at a given iteration is defined as $s_r = \frac{n_p}{n_u}$, where $n_p$ is the number of weights pruned at that iteration and $n_u$ is the number of weights that remain unpruned before pruning. As the number of unpruned weights decreases over the iterations, the absolute number of weights removed at each step decreases accordingly. The relative pruning rate is treated as a training hyperparameter and can be tuned alongside the other hyperparameters.\\
The number of pruning iterations, $N$, is chosen such that the desired final sparsity is reached while keeping the relative pruning rate per iteration as close as possible to the specified value. After each pruning step, the remaining weights are retained and used to initialize the next training iteration, rather than reinitializing the network from scratch. The iterative weight pruning routine used in this work is described in Algorithm~\ref{alg:pruning_weights}.  
Let $\boldsymbol{\theta}$ denote the network parameters and 
$M \in \{0,1\}^{|\boldsymbol{\theta}|}$ a binary mask indicating which parameters are active. 
The network output is evaluated as $f(X; M \odot \boldsymbol{\theta})$, where 
$X$ denotes the training data and $\odot$ the element-wise product. 
At each iteration, a scoring function $S = \textit{score}(M \odot \boldsymbol{\theta})$ 
evaluates the absolute values of the currently active weights and selects the smallest ones until the target relative pruning rate $s_r$ is reached.
The pruning function $\textit{prune}(M,S)$ then updates the mask $M$ by removing parameters with low scores according to the chosen pruning strategy.
Once the desired final sparsity has been attained, an optional fine-tuning step can be performed to recover predictive accuracy. Depending on the application, however, this step is not always necessary and may even degrade performance~\cite{Liu2018_RethinkingValueNetwork, Plate2025_analysisoptimizationproblems}.

\begin{algorithm}[H]
\caption{Iterative magnitude-based weight pruning}
\label{alg:pruning_weights}
\begin{algorithmic}[1]
\REQUIRE $s_f$ (target sparsity), $s_r$ (relative pruning rate), $X$ (training dataset)
\STATE $N \leftarrow \textit{calculateNumIterations}(s_f, s_r)$
\STATE $\boldsymbol{\theta} \leftarrow \textit{xavierInitialization}()$
\STATE $\boldsymbol{M} \leftarrow \mathbf{1}^{|\boldsymbol{\theta}|}$
\FOR{$i \leftarrow 1$ \TO $N$}
    \STATE $\boldsymbol{\theta} \leftarrow \textit{trainToConvergence}(f(X; \boldsymbol{M} \odot \boldsymbol{\theta}))$
    \STATE $S \leftarrow \textit{score}(\boldsymbol{M} \odot \boldsymbol{\theta}, s_r)$
    \STATE $\boldsymbol{M} \leftarrow \textit{prune}(\boldsymbol{M}, S)$
\ENDFOR
    \IF{fine-tuning is desired}
        \STATE $\boldsymbol{\theta} \leftarrow \textit{fineTune}(f(X; \boldsymbol{M} \odot \boldsymbol{\theta}))$
    \ENDIF
\RETURN $\boldsymbol{M}, \boldsymbol{\theta}$
\end{algorithmic}
\end{algorithm}

\subsubsection{Cleaning dead neurons}
When pruning weights from a network, it may happen that certain neurons no longer influence the network output and can therefore be removed entirely. This can occur, for example, when all incoming weights to a neuron have been pruned (Figure~\ref{fig:dead_neuron_single}), so that its activation becomes constant, or when all outgoing weights have been removed, so that the neuron no longer affects subsequent layers. We refer to such neurons as \textit{dead neurons}.\\
If all outgoing weights of a neuron have been removed, the neuron can be deleted directly, as it is disconnected from the remainder of the network. The more relevant case considered here is when all incoming weights to a neuron have been pruned. In that situation, the neuron no longer depends on the network input, yet it still produces a constant output given by its bias term. To remove the neuron without altering the network function, this constant contribution must be transferred to the biases of the neurons in the subsequent layer. This principle has been discussed in the broader literature on neural network simplification and compression~\cite{Vecoven2020_Introducingneuromodulationdeep}. The resulting bias update for each connected neuron in layer $j+1$ is
\begin{equation}
    \label{eq:dead_neuron_bias}
    b_i^{\prime(j+1)} = b_i^{(j+1)} + w^{(j+1)}_{d,i} \cdot \sigma(b_d^{(j)}),
\end{equation}
Equation~\eqref{eq:dead_neuron_bias} defines the adjusted bias term $b_i^{\prime(j+1)}$ for a neuron $i$ in the subsequent layer $j+1$ of a dead neuron in layer $j$ having index $d$. There, $w^{(j+1)}_{d,i}$ is the output weight connecting the dead neuron to a neuron $i$ in the subsequent layer, $b_i^{(j+1)}$ represents its original bias term and $\sigma(b_d^{(j)})$ is the constant output of the dead neuron, where $\sigma$ is the ReLU activation function in this case. Figure~\ref{fig:illustrative_dead_neuron_cleaning} illustrates the dead-neuron cleaning process, in which the constant activation of a dead neuron (index $0$) is transferred to the biases of three subsequent neurons through its output weights. The corresponding output weights ($w_{0,1}$, $w_{0,2}$, $w_{0,3}$) can then be safely pruned from the network.
\begin{figure}[t]
    \centering
    \begin{subfigure}[b]{0.3\linewidth}
        \centering
        \includegraphics[width=\linewidth]{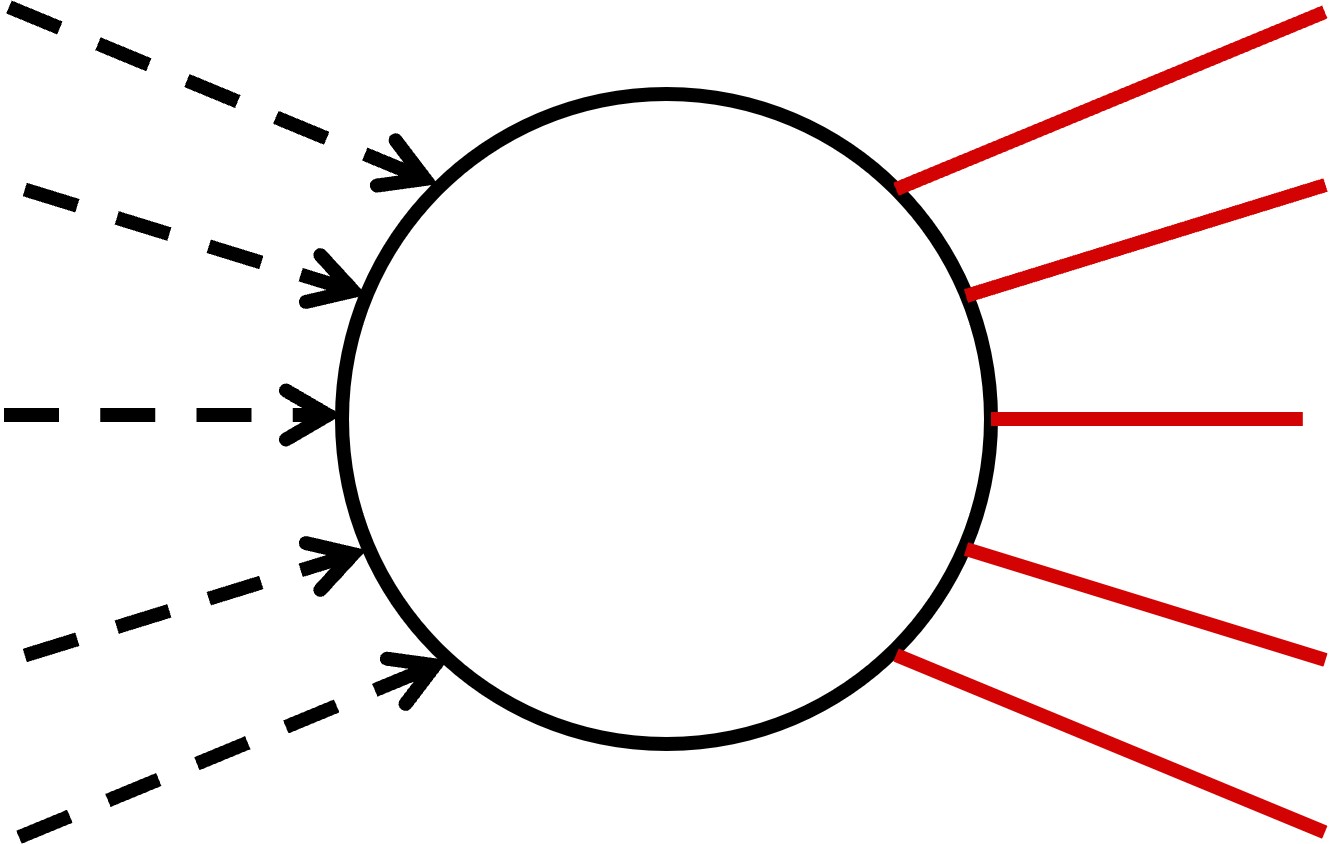}
        \caption{}
        \label{fig:dead_neuron_single}
    \end{subfigure}
    \hspace{0.1\linewidth}
    \begin{subfigure}[b]{0.3\linewidth}
        \centering
        \includegraphics[width=\linewidth]{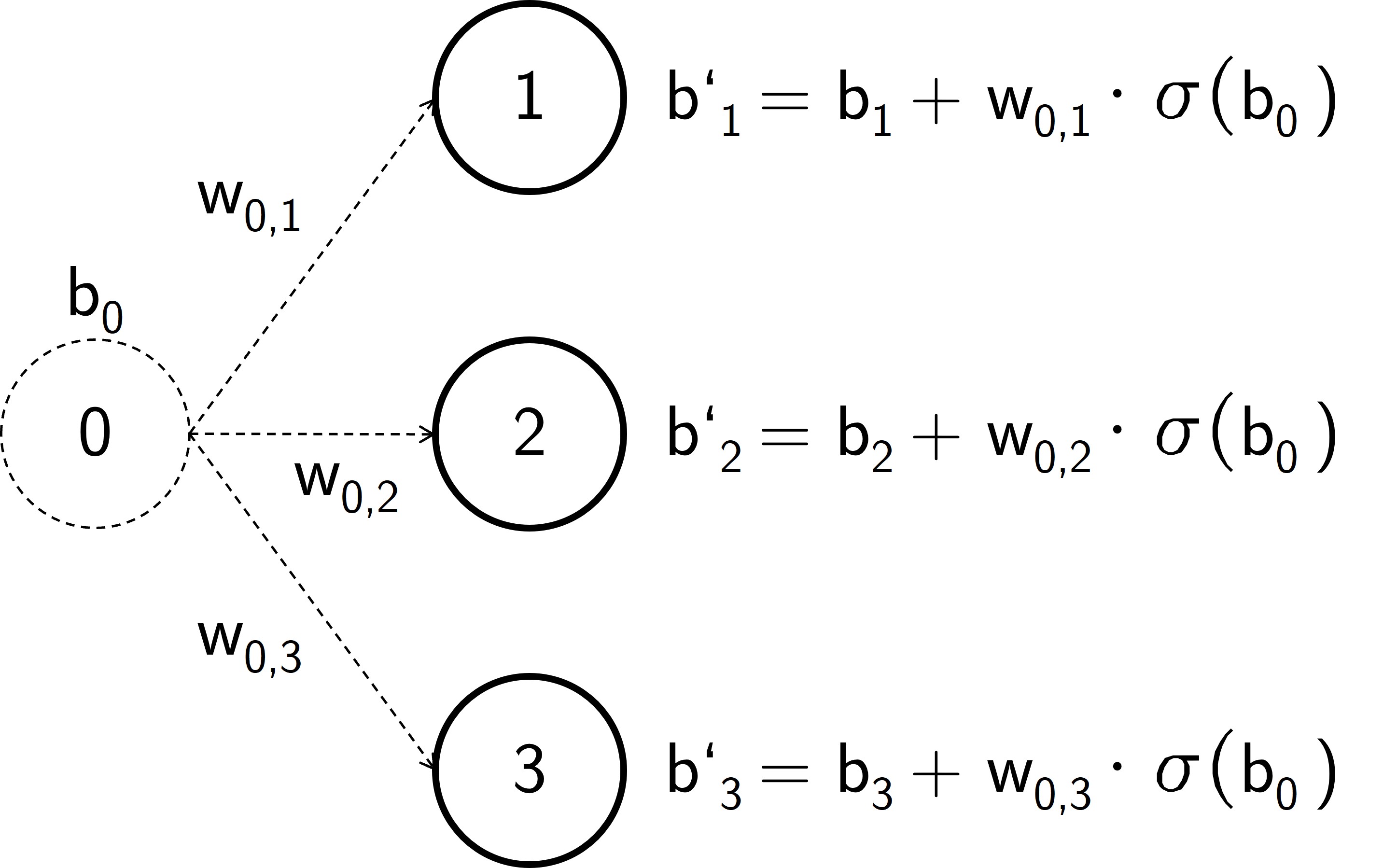}
        \caption{}
        \label{fig:illustrative_dead_neuron_cleaning}
    \end{subfigure}
    \caption{(a) Illustration of a dead neuron with pruned incoming weights and consequent flat activation and (b) an illustrative example for the dead neuron cleaning, where the biases of the neurons in the subsequent layer are adjusted to account for the constant activation of the dead neuron, with consequent safe removal of the output weights.}
    \label{fig:dead_neuron_cleaning}
\end{figure}

\subsubsection{Node pruning}
\textit{Node pruning} is a structured pruning method for fully-connected NNs, where entire nodes are removed from the network.
The effectiveness of tailored structured pruning for subsequent network embedding in optimization problems has been demonstrated by Cacciola et al.~\cite{Cacciola2023_StructuredPruningNeural}.
In this work, we aim to complement previous research with a mechanistic interpretation of its effectiveness supported by empirical evidence and comparison across different baselines.
In node pruning, entire nodes are removed from the network.
This is done by cutting all incoming weights and bias of a neuron, as well as all weights connected to the output of the neuron.
Similarly to the weight pruning case, we employ a magnitude-based technique that removes nodes with the smallest sum of absolute values of their input weights.
Given the nature of the ReLU activation function, pruning nodes with small input weights eliminates neurons producing near-zero activations, which should minimally affect the network predictions.
To preserve the relative widths of the hidden layers, node pruning is performed independently within each layer using the same sparsity level. In other words, the same fraction of nodes is removed from every hidden layer.\\
The algorithm to perform node pruning is equivalent to the one for weight pruning (Algorithm~\ref{alg:pruning_weights}). However, the definition of sparsity differs in this context, as it is considered on a node-wise basis. For instance, $s_f$ should be interpreted as the ratio of pruned nodes to the total number of nodes. Consequently, a moderate level of node sparsity may correspond to a high level of weight sparsity. Moreover, the \textit{prune} function accounts for layer-wise constant sparsity. Nodes are then pruned by masking out all input weights, output weights, and the bias term.\\
\textbf{Illustrative example.}
Consider a fully connected network with architecture $2\!-\!5\!-\!5\!-\!1$.  
The total number of weights is $40$, without accounting for bias terms.
With $20\%$ \emph{weight pruning}, we remove $8$ weights, leaving $32$ weights (a $20\%$ reduction).
With $20\%$ \emph{node pruning} applied evenly to each hidden layer, one neuron is removed from each $5$-neuron hidden layer, giving a new architecture $2\!-\!4\!-\!4\!-\!1$.  
The remaining number of weights becomes $28$.
Thus $12$ weights are removed, corresponding to a $30\%$ reduction.  

\subsubsection{Regularization}
\label{sec:regularization}
Regularization is used to facilitate pruning.
Regularization factors in the loss function, such as $L_1$ or $L_2$, are commonly introduced to avoid overfitting and increase the generalization performance of a neural network~\cite{Kukacka2017_RegularizationDeepLearning}.
Here, the idea of using regularization in combination with pruning is to promote sparsity in the weight matrices~\cite{Cacciola2023_StructuredPruningNeural}.
For instance, $L_1$ regularization (i.e., $R(w) = \|W\|_1$) can drive entries of the weight matrices $W$ toward zero.
Sparse weight matrices allow for removing parameters without affecting the model output.
The literature shows no clear consensus between $L_1$ and $L_2$ regularization in pruning, as some studies report better results with $L_1$, while others favor $L_2$~\cite{Han2015_LearningbothWeights, Xiao2018_TrainingFasterAdversarial, Zhang2018_SystematicDNNWeight}.
Besides, a variety of different regularization techniques have been used in the past to support pruning~\cite{LambertLacroix2016_adaptiveBerHupenalty, Louizos2017_LearningSparseNeural, Chen2021_OnlyTrainOnce}.
Furthermore, researchers developed tailored regularization techniques to induce structured sparsity~\cite{Wen2016_LearningStructuredSparsity, Cacciola2022_DeepNeuralNetworks}.
In particular, \cite{Cacciola2022_DeepNeuralNetworks} introduced Structured Perspective Regularization (SPR), a theoretically motivated regularization term derived from a MIP formulation of the joint training-pruning problem. The purpose of SPR is to promote structured sparsity, and in particular node-level sparsity, during training.
Given a weight tensor \(W\), the SPR penalty, balanced by a multiplier $\lambda$, is defined as
\begin{equation}
\label{eq:SPR}
z(W;\alpha,M)=
\begin{cases}
2\sqrt{(1-\alpha)\alpha}\,\|W\|_{2},
&
\text{if }
\dfrac{\|W\|_{\infty}}{M}
\leq
\sqrt{\dfrac{\alpha}{1-\alpha}}\,\|W\|_{2}
\leq 1,
\\[1.2em]
\dfrac{\alpha M}{\|W\|_{\infty}}\|W\|_{2}^{2}
+
(1-\alpha)\dfrac{\|W\|_{\infty}}{M},
&
\text{if }
\sqrt{\dfrac{\alpha}{1-\alpha}}\,\|W\|_{2}
\leq
\dfrac{\|W\|_{\infty}}{M}
\leq 1,
\\[1.2em]
\alpha\|W\|_{2}^{2}+(1-\alpha),
&
\text{otherwise}.
\end{cases}
\end{equation}

Here, \(M\) is a constant, \(\alpha\) is a tunable hyperparameter, and \(W\) denotes the weight tensor associated with the structure to be pruned, such as the incoming weight tensor of a neuron.

\subsection{Big-M formulation of ReLU networks}
\label{sec:bigM_formulation}
NNs are often used as surrogate models to replace complex or unknown constraints in optimization problems.
This method is also known as constraint learning, since a learning algorithm approximates optimization constraints.
A general form of optimization problems with neural network surrogates can be formulated as follows:
\begin{equation}
    \label{eq:opt_problem_generalform}
    \begin{aligned}
    &\min_{\textbf{x}\in \mathcal{D}, \textbf{y}\in \mathbb{R}^{N_O}, \textbf{t} \in \mathcal{T}} f(\textbf{x}, \textbf{y}, \textbf{t})\\
    &\text{s.t.} \quad \textbf{y} = h_\mathbf{\theta}(\textbf{x})\\
    & \quad \quad \:\, \textbf{g}(\textbf{x},\textbf{y}, \textbf{t})\leq \textbf{0}.
    \end{aligned}
\end{equation}
Here $\textbf{x}\in \mathbb{R}^{N_i}$, bounded in a domain $\mathcal{D}=[\textbf{x}^L,\textbf{x}^U]$, is the input to the neural network. Together with additional variables $t \in \mathcal{T}$, they define the degrees of freedom of the problem. $h_\mathbf{\theta}$ is a neural network with trained parameters $\mathbf{\theta}$ approximating a (set of) equality constraints $h_\theta: \mathcal{D} \mapsto \mathbb{R}^{N_O}$. $\mathbf{y}$ is an optimization variable introduced to represent the neural network output. Lastly, $\textbf{g} : \mathcal{D} \times \mathbb{R}^{N_O} \times \mathcal{T} \mapsto \mathbb{R}^{N_C}$ 
represents a set of inequality constraints and 
$f : \mathcal{D} \times \mathbb{R}^{N_O} \times \mathcal{T} \mapsto \mathbb{R}$ is the objective function 
of the optimization problem. Note that equality constraints can also be represented in this formulation 
by combining two inequalities of the form $g_i(\textbf{x},\textbf{y},\textbf{t}) \le 0$ and 
$-g_i(\textbf{x},\textbf{y},\textbf{t}) \le 0$.\\
To enable the use of deterministic global optimization methods, neural networks are often reformulated algebraically by representing their operations as a system of equations $\mathbf{h}(\mathbf{x},\mathbf{z})$. This approach introduces additional variables $\textbf{z}$ corresponding to the internal activations of the network and allows the resulting model to be handled in an equation-oriented environment. Alternative approaches exist that avoid an explicit system representation and instead operate directly on the neural network function (\textit{reduced-space})~\cite{Schweidtmann2018_DeterministicGlobalOptimization,Schweidtmann2022_OptimizationTrainedMachine} or treat the network as a black-box model~\cite{ElorzaCasas2025_comparisonstrategiesembed}. In this work, we adopt the commonly used \textit{full-space} (FS) representation, thereby introducing a new set of bounded variables $\textbf{z} \in \mathcal{Z}$ representing the internal neural network activation values.
\begin{equation}
    \label{eq:fs-formulation}
    \begin{aligned}
    &\min_{\textbf{x}\in \mathcal{D}, \textbf{y}\in \mathbb{R}^{N_O}, \textbf{t} \in \mathcal{T}, \textbf{z} \in \mathcal{Z}} f(\textbf{x}, \textbf{y}, \textbf{t})\\
    &\text{s.t.} \quad \textbf{y} - \textbf{h}(\textbf{x}, \textbf{z}) = 0\\
    & \quad \quad \:\, \textbf{g}(\textbf{x},\textbf{y},\textbf{t})\leq \textbf{0}.
    \end{aligned}
\end{equation}
In this formulation (Eq.~\eqref{eq:fs-formulation}), the input $\textbf{x}$, internal neural network variables $\textbf{z} \in \mathcal{Z}$, and output $\textbf{y}$ are optimization variables, and it is therefore referred to as a full-space (FS) formulation. FS formulations are commonly used in deterministic global solvers, and most B\&B algorithms require bounds on $\textbf{x}$, $\textbf{z}$, and $\textbf{y}$~\cite{Schweidtmann2018_DeterministicGlobalOptimization}. In the FS setting, the algebraic structure of $\mathbf{h}(\mathbf{x},\mathbf{z})$ depends on the activation functions used in the neural network. In the following, we focus on the rectified linear unit (ReLU), which remains one of the most widely used activation functions in practice due to its favorable computational properties and effective gradient propagation~\cite{Glorot2011_DeepSparseRectifier, Kunc2024_ThreeDecadesActivations}. ReLU is defined as
\begin{equation}
    \label{eq:relu}
    \text{ReLU}(x)=\max(0,x),
\end{equation}
Hence, it is not continuously differentiable, but piecewise linear and still continuous at $x=0$. Different methods have been proposed to formulate the activation-deactivation behavior algebraically~\cite{Huchette2023_WhenDeepLearning}. Nevertheless, the big-M formulation for ReLU networks remains the most widely used in literature~\cite{Grimstad2019_ReLUnetworksasa}. Big-M is a standard technique to model \textit{if-else} statements in optimization problems by introducing integer variables $\delta_i \in \{0,1\}$. Therein, a single neuron $i$ in layer $j$, $z_i^{(j)} = \text{ReLU} (W_i^{(j)} \textbf{z}^{(j-1)}+b_i^{(j)})$, can be formulated as a mixed-integer linear program (MILP), where we define $\textbf{x} = \textbf{z}^{(0)}$ and $\textbf{y} = \textbf{z}^{(N_L)}$ in the first and last layer of the NN, respectively:
\begin{equation}
    \label{eq:bigMformulation}
    \begin{aligned}
    & z_i^{(j)} \geq 0,\\
    & z_i^{(j)} \geq W_i^{(j)} \textbf{z}^{(j-1)} + b_i^{(j)},\\
    & z_i^{(j)} \leq W_i^{(j)} \textbf{z}^{(j-1)} + b_i^{(j)} - L_i^{(j)} (1-\delta_i^{(j)}),\\
    & z_i^{(j)} \leq U_i^{(j)} \delta_i^{(j)},\\
    & j=1,...,N_L, \quad i = 0, ..., N_N^{(j)} 
    \end{aligned}
\end{equation}
where $W_i^{(j)}$ is the $i$-th row of the weight matrix in the layer $j$, $b_i^{(j)}$ is the bias term of neuron $i$ in layer $j$, $\textbf{z}^{(j-1)}$ is the vector of incoming activation value, $N_L$ is the number of neural network layers and $N_N^{(j)}$ is the number of neurons in layer $j$.
A major role for the convergence speed of deterministic solvers is played by the integer variables $\delta_i^{(j)}$ (one for each neuron) and big-M coefficients $L_i^{(j)}$ and $U_i^{(j)}$.
In fact, the formulation assumes a bounded preactivation $L_i^{(j)} \leq W_i^{(j)} \textbf{z}^{(j-1)} + b_i^{(j)} \leq U_i^{(j)}$ that deterministic solvers use to build continuous relaxations.
The choice of broad bounds leads to weaker relaxations which make the feasible region unnecessarily large and the convergence of B\&B methods slow. 
Various stronger MILP formulations for ReLU networks based on big-M have been proposed in the literature~\cite{Anderson2020_Strongmixedinteger, Tsay2021_PartitionBasedFormulations, Huchette2023_NonconvexPiecewiseLinear}. 
Nevertheless, the choice of big-M values remains crucial and has motivated the development of techniques to tighten the corresponding interval bounds.

\subsection{Bound tightening}
\label{sec:bound_tightening}
Techniques aiming to refine the choice of variable bounds, such as big-M coefficients, are known as \textit{bound tightening} methods.
The main challenge in ReLU networks lies in the need to propagate known bounds on the input variables ($L^{(0)}=\textbf{x}^L$ and $U^{(0)}=\textbf{x}^U$) through successive layers to determine bounds $L_i^{(j)}$ and $U_i^{(j)}$ for each preactivation value such that $L_i^{(j)} \leq W_i^{(j)} \textbf{z}^{(j-1)} + b_i^{(j)} \leq U_i^{(j)}$.\\
To compute these bounds, different bound tightening strategies can be used, ranging from inexpensive propagation methods to more expensive optimization-based procedures. A basic approach is \textit{interval arithmetic}~(IA), which propagates input bounds through the network layer by layer. This yields valid but often weak estimates, since it assumes worst-case combinations and ignores dependencies between variables~\cite{Plate2025_analysisoptimizationproblems}. Stronger bounds can be obtained by propagating more informative relaxations than simple intervals, for example based on convex under- and overestimators such as McCormick or $\alpha\text{BB}$ relaxations~\cite{Schweidtmann2018_DeterministicGlobalOptimization, Androulakis1995_BBglobal}. Such methods remain computationally tractable, but typically capture fewer interactions than full optimization-based approaches. 
Optimization-based bound tightening~(OBBT) can further tighten the bounds by solving additional sub-optimization problems for each neuron. Specifically, the preactivation values $W_i^{(j)} \textbf{z}^{(j-1)} + b_i^{(j)}$ are minimized and maximized, respectively, subject to the neural network formulation, while treating the internal neural network variables $(\textbf{z}, \boldsymbol{\delta})$ as degrees of freedom. Typically, the integer variables are relaxed, yielding a set of more tractable LPs (LP-based bound tightening).\\
Although LP-based bound tightening estimates stronger bounds, reflecting in more efficient subsequent optimization, the cost of solving the LP sub-problems is not negligible, and scales badly with deep and large NNs (i.e., number of LP sub-problems = $2 \sum_j N_N^{(j)}$).
Moreover, the efficiency of OBBT is affected by the quality of the initial bounds provided by IA.
In practice, deterministic global solvers employ a combination of bound tightening strategies, often supported by additional heuristics.
The reader is also referred to the survey from Puranik and Sahinidis~\cite{Puranik2017_Domainreductiontechniques} for a complete overview of bound tightening techniques for global optimization.

\subsection{Related works}
\label{sec:related_works}
A few recent studies have examined how network sparsity and training choices affect optimization problems with embedded ReLU networks.
Cacciola et al.~\cite{Cacciola2023_StructuredPruningNeural} showed that pruning a network before embedding it in a MIP substantially accelerates its solution, and argued that structured sparsity is the most suitable form of compression in this context. Their study develops and demonstrates an advanced structured method based on the Structured Perspective Regularization~\cite{Cacciola2022_DeepNeuralNetworks}, evaluated on adversarial-example generation for MNIST classifiers.
Pham et al.~\cite{Pham2025_OptimizationoverTrained} addressed a complementary question: given a large, predetermined network, they use a pruned counterpart as a \emph{surrogate within a surrogate}, i.e., they solve the sparse model and transfer the resulting solutions back to the original dense network. Their broad computational benchmark spans network verification and function maximization, and reports that skipping fine-tuning is often cost-effective for this purpose.
Plate et al.~\cite{Plate2025_analysisoptimizationproblems} focused primarily on how training and post-training choices, including $L_1$ regularization, clipped ReLUs, dropout, optimization-based bound tightening, and a novel equivalent scaling transformation, influence the number of linear regions, the fraction of stable neurons, and solver runtime for deterministic minimization of trained ReLU surrogates on synthetic benchmark functions.
These works provide the foundation for the present study; Table~\ref{tab:related_works} summarizes how their scope relates to ours.

\begin{table}[htbp]
\centering
\small
\setlength{\tabcolsep}{4pt}
\caption{Positioning of the present study with respect to the closest related works in the literature.}
\label{tab:related_works}
\begin{tabular}{p{1.6cm}p{3.3cm}p{3.3cm}p{3.3cm}p{3.4cm}}
\toprule
Work & Primary question & Pruning strategies & Optimization setting and evidence & Distinctive focus \\
\midrule
Cacciola et al.~\cite{Cacciola2023_StructuredPruningNeural}
& Can specialized structured pruning make constraint-learning MIPs tractable?
& SPR-based structured pruning with fine-tuning
& Big-$M$ MILP, Gurobi; adversarial generation on MNIST classifiers
& Develops a dedicated structured pruning method and demonstrates the speedup it enables \\
\addlinespace
Pham et al.~\cite{Pham2025_OptimizationoverTrained}
& Can a sparse network act as a surrogate for optimizing over a given dense network?
& Unstructured and structured, magnitude-based and random, with and without fine-tuning
& MILP, Gurobi; verification on MNIST and Fashion-MNIST, function maximization
& Transfers solutions from the sparse model back to the dense model, within a time limit \\
\addlinespace
Plate et al.~\cite{Plate2025_analysisoptimizationproblems}
& How do training and post-training choices affect optimization over trained ReLU networks?
& Not pruning; $L_1$ regularization, clipped ReLU, dropout, OBBT, equivalent scaling
& Big-$M$ MILP, Gurobi; three two-dimensional benchmark functions
& Relates optimization-aware training and scaling choices to linear regions, stable neurons and runtime \\
\addlinespace
This work
& When, and through which formulation-level mechanisms, does pruning accelerate deterministic global optimization of trained ReLU networks?
& Iterative magnitude-based unstructured and structured pruning; SPR-based structured pruning
& Big-$M$ MILP, commercial (Gurobi) and open-source (HiGHS) solvers; >1,000 networks across depth--width and size sweeps on a test function and a chemical process flowsheet surrogate
& Identifies the formulation-level mechanisms through which pruning accelerates the MILP, using a controlled benchmark of pruning strategies across architectures, techniques, and sparsity levels \\
\bottomrule
\end{tabular}
\end{table}

The present work is complementary to these studies. We do not propose a new pruning algorithm. Instead, we investigate through which mechanism pruning alter the resulting optimization formulation and solver performance. In particular, we study the combined effects of different pruning strategies on network approximation accuracy, optimization problem formulation, and deterministic global solution time across different network complexities and morphologies, dimensions that were not considered jointly in the earlier studies. We further evaluate whether these mechanisms persist in a chemical process surrogate optimization problem and across a commercial and an open-source MILP solver.

\section{Effect of pruning on ReLU big-M formulations}
\label{sec:methods}
This section analyzes how pruning affects optimization problems with embedded ReLU networks formulated through big-$M$ constraints. 
First, we discuss how pruning-induced sparsity can impact the main stages of MILP solvers. 
Then, we examine how pruning modifies the resulting big-$M$ formulation, focusing on problem-size reduction, binary-variable reduction, and bound tightening. 
These formulation-level effects provide the basis for interpreting the computational results presented in Section~\ref{sec:experiments}.\\

\subsection{How pruning can impact mixed-integer linear programming solvers}
\label{sec:milp_solvers}
Most commercial solvers, such as CPLEX, GUROBI~\cite{GurobiOptimization2024_GurobiOptimizerReference}, and XPRESS, and some of the noncommercial ones (e.g., SCIP~\cite{Bolusani2024_SCIPOptimizationSuite}), implement sophisticated \textit{branch and cut} methods yielding global optima for MILPs. 
In the branch and cut method, the B\&B tree search is supported by the introduction of additional cuts to strengthen the lower bound, thereby avoiding the exhaustive enumeration of all integer combinations~\cite{Grossmann2021_AdvancedOptimizationProcess}.
We identify three main stages in MILP solvers: (1) preprocessing, (2) lower bounding, and (3) upper bounding. In this work, we analyze how pruning-induced sparsity affects each stage.\\
During preprocessing, MILP codes employ techniques to reduce the problem size and provide bounds on the internal variables. We will discuss how pruning affects bound tightening techniques in Section~\ref{sec:bound_tightening_pruning}.
Sparsity affects the problem size reduction in pre-processing.
We analyze this empirically in our experiments (see Section~\ref{sec:analysis_experiments_formulation}).
The lower bound of the problem is found by solving the continuous LP relaxation at every binary node of the search tree.
In branch and cut methods, the lower bound is strengthened by adding cutting planes (e.g., Gomory's method~\cite{Gomory1960_algorithmmixedinteger, Gomory2009_OutlineAlgorithmInteger}), which introduce additional linear inequalities to obtain a tighter relaxation.
Hence, the lower bound is found by iteratively solving a sequence of LPs while adding cutting planes.
Sparsity positively impacts the solution of the relaxed LP problems commonly tackled through the simplex method.
In particular, pruning entire rows and columns in the weight tensors results in smaller basis matrices that are less expensive to factorize and invert, whereas the random sparsity introduced by weight pruning can be leveraged in linear algebra routines to reduce factorization cost.
The generation of cutting planes exploits information from the optimal simplex tableau, hence we do not expect an impact of sparsity in this phase~\cite{Bertsimas1997_IntroductionLinearOptimization}.
The upper-bounding step is usually based on heuristics such as rounding of fractional variables, feasibility pump, or local searches to find integer-feasible solutions~\cite{Bertsimas1997_IntroductionLinearOptimization}. Pruning may improve the efficiency of this stage when such heuristics rely on solving auxiliary LPs or sub-MILPs, since smaller and sparser formulations can again reduce computational cost. Nevertheless, because upper-bounding performance depends strongly on the specific heuristic and solver implementation, we expect the effect of pruning in this stage to be less systematic than in preprocessing or lower bounding.

\subsection{Big-M formulation of pruned networks}
\label{sec:formulation_pruned}
The complexity of a MIP depends on (1) the problem size (i.e., the number of variables and constraints), (2) the number of integer variables, i.e., the combinatorial complexity, and (3) the bounds tightness, i.e., the strength of the relaxation of the integer variables.
In the following, we discuss how pruning affects the big-M formulation of ReLU networks, thereby impacting each of the aforementioned factors and explaining the resulting increase in computational efficiency.

\subsubsection{Pruning leads to smaller and simpler problems}
\label{sec:problem_formulation_sparsity}
The formulation of pruned NNs is affected by the sparsity induced by pruning itself.
To assess its impact on the computational runtime, it is necessary to understand how different \textit{types} of sparsity induced by pruning influence each stage of the solver.
In weight pruning, some of the coefficients of the weight matrix row $W_i^{(j)}$ (cf. Eq.~\eqref{eq:bigMformulation}) may be masked.
Hence, the second and third inequality constraints in Eq.~\eqref{eq:bigMformulation} will ultimately depend on a subset of the incoming activation values $\mathbf{z}^{(j-1)}$.
The optimization solver will thereby receive a randomly sparse constraint matrix.
More interestingly, in node pruning, entire rows of the matrix $W^{(j)}$ and corresponding columns of the subsequent matrix $W^{(j+1)}$ are set to zero (cf. Figure 3 in~\cite{Cacciola2023_StructuredPruningNeural} for an explicative illustration).
This structured sparsity leads to the elimination of entire constraints from the neural network's algebraic formulation.
The solver can then delete the related activation $z_i^{(j)}$ and, more importantly, the integer variable associated with the node that has been pruned.\\
Random sparsity induced by weight pruning can be exploited in linear algebra routines to reduce the cost of basis-matrix factorization and the associated linear solves arising in the solution of the relaxed LP problems.
Conversely, structured sparsity from node pruning (1) reduces the number of variables and constraints, thereby affecting LP solution time, and (2) eliminates integer variables, which directly decreases the combinatorial complexity of the problem and the number of branching steps required for the B\&B algorithm to converge.
This also corresponds to a reduction in the number of linear regions of the ReLU network, a factor identified by Plate et al.~\cite{Plate2025_analysisoptimizationproblems} as influencing the complexity of optimization problems with embedded ReLU networks.
As a final, relevant note, we observe that entire nodes may be pruned as a side effect of weight pruning, especially at high sparsity rates.
Therefore, the benefits observed for node pruning may likewise extend to weight pruning.

\subsubsection{Bounds tightening in pruned networks}
\label{sec:bound_tightening_pruning}
We show that pruned ReLU networks yield big-M formulations with systematically tighter bounds on neuron preactivations. 
This matters because the big-M constants $L_i$ and $U_i$ (bounding each preactivation $p_i = W_i\mathbf{z}+b_i$) directly shape the continuous relaxation solved inside B\&B: smaller intervals $[L_i,U_i]$ typically produce a stronger relaxation, hence better lower bounds and fewer nodes to explore~\cite{Puranik2017_Domainreductiontechniques}.
Intuitively, tighter bounds reduce the “vertical slack” available to the ReLU constraints, which tightens the convex relaxation of the piecewise-linear graph (Fig.~\ref{fig:relaxation}) and potentially accelerates convergence.
\begin{figure}
    \centering
    \includegraphics[width=0.5\linewidth]{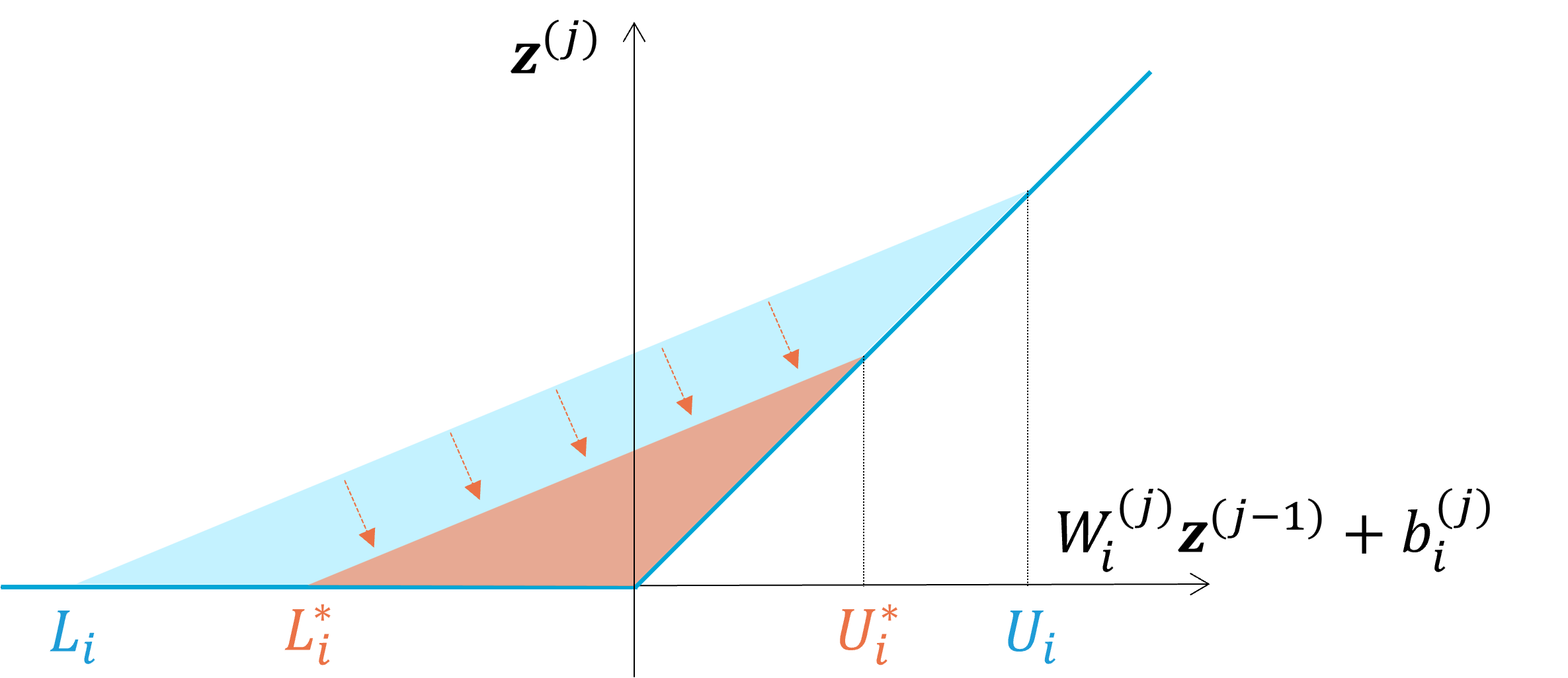}
    \caption{Bound tightening and resulting relaxation of ReLU.}
    \label{fig:relaxation}
\end{figure}
MILP solvers typically compute such bounds during variables definition and preprocessing.
In ReLU networks, the input bounds $\mathbf{z}^{(0)}$ are often given, e.g., as a box or as a convex-hull-based enclosure of observed data~\cite{Courrieu1994_Threealgorithmsestimating, Schweidtmann2021_Obeyvaliditylimits}.
These input bounds are then propagated forward to infer bounds for each neuron.
A common inexpensive method is interval arithmetic (cf. Sect.~\ref{sec:bound_tightening}), which propagates lower/upper bounds layer by layer:
\begin{align}
L_i^{(j)} &=
\sum_{k=1}^{N_N^{(j-1)}} \min\left\{W_{i,k}^{(j)} L_k^{(j-1)},\; W_{i,k}^{(j)} U_k^{(j-1)}\right\}
+ b_i^{(j)},
\quad j=1,\dots,N_L, \quad i=1,\dots,N_N^{(j)} \label{eq:lower_bound-IA} \\
U_i^{(j)} &=
\sum_{k=1}^{N_N^{(j-1)}} \max\left\{W_{i,k}^{(j)} L_k^{(j-1)}, W_{i,k}^{(j)} U_k^{(j-1)}\right\}
+ b_i^{(j)},
\quad j=1,\dots,N_L,\quad i=1,\dots,N_N^{(j)} \label{eq:upper_bound-IA}
\end{align}
IA always yields valid bounds but can be conservative because it ignores the interdependency between upstream active neurons~\cite{Pham2025_OptimizationoverTrained}.\\
We observe that pruning can partially alleviate the overestimation of bound widths computed through IA. 
We formalize this IA-monotonicity property below. 
The key observation is that, under interval arithmetic, affine-layer widths propagate through the absolute values of the weights, while ReLU activations cannot increase interval widths.

\begin{proposition}[IA-monotonicity property under pruning]
\label{prop:monotonic-tightening-short}
Consider a feedforward ReLU network with preactivations
$\mathbf{p}^{(j)} = W^{(j)}\mathbf{z}^{(j-1)} + \mathbf{b}^{(j)}$ and activations
$\mathbf{z}^{(j)} = \mathrm{ReLU}(\mathbf{p}^{(j)})$.
Let $\Delta \mathbf{p}^{(j)}$ and $\Delta \mathbf{z}^{(j)}$ denote the interval-arithmetic widths of the preactivation and activation bounds, respectively. 
Let $W'^{(j)}$ be obtained from $W^{(j)}$ by setting an arbitrary subset of weights to zero, while keeping the remaining weights fixed. 
Then, for every layer $j$,
\[
\Delta \mathbf{p}'^{(j)} \leq \Delta \mathbf{p}^{(j)}, 
\qquad
\Delta \mathbf{z}'^{(j)} \leq \Delta \mathbf{z}^{(j)}
\]
componentwise. 
Consequently, the IA-based big-$M$ widths $U_i^{(j)}-L_i^{(j)}$ are never larger after pruning.
\end{proposition}


\noindent Proposition~\ref{prop:monotonic-tightening-short} establishes a no-regret property of the pruning step itself: when a subset of weights is set to zero while all remaining weights are kept fixed, the resulting IA-based big-$M$ intervals are at least as tight as before. A proof of the proposition is provided in Appendix~\ref{app:proofs}.
This result does not directly extend to the full iterative pruning algorithm, since retraining and optional fine-tuning can modify the remaining weights and may, in principle, widen the IA bounds again. 
Nevertheless, our empirical results show that the overall tightening trend persists in practice, even when retraining and fine-tuning are performed (see Section~\ref{sec:experiments}).\\
Moreover, the previous result is non-strict: pruning cannot worsen IA widths, but it may leave them unchanged. 
A strict decrease occurs when the pruned connection carries nonzero interval variation and this variation reaches the output through a non-degenerate active path.

\begin{proposition}[Strict IA tightening under non-degeneracy]
\label{prop:strict-tight-short}
Under the setting of Proposition~\ref{prop:monotonic-tightening-short}, suppose that a nonzero weight 
$W_{i,k}^{(j^\star)}$ is pruned, the corresponding upstream activation has nonzero interval width, 
$\Delta z_k^{(j^\star-1)} > 0$, and neuron $i$ influences an output coordinate through at least one intact downstream path. 
If no ReLU on this path is degenerate, i.e., always inactive or zero-width over the considered domain, then the IA width of that output coordinate is strictly reduced:
\[
\Delta z'_{\mathrm{out}} < \Delta z_{\mathrm{out}}.
\]
\end{proposition}


We prove the proposition in Appendix~\ref{app:proofs}. Following the proposition, the pruning step reduces absolute weight magnitudes in the IA propagation, which narrows preactivation intervals; ReLU then cannot expand these intervals, so the shrinkage persists, and under mild non-degeneracy assumptions, becomes strict and propagates to the output.
In turn, narrower preactivation bounds translate into smaller big-$M$ gaps and a tighter LP relaxation, which is leveraged by B\&B for fast convergence.
In Section~\ref{sec:analysis_experiments_formulation}, we report substantial reductions in bound widths and LP relaxation gaps for pruned networks, even after fine-tuning, although this condition is not guaranteed in general.\\
While IA provides computationally efficient bound estimates, LP-based bound tightening can further strengthen these bounds by solving optimization sub-problems for each neuron.
However, the computational cost of LP-based methods scales poorly with network size.
For pruned networks, the monotonic IA tightening shown above already improves the relaxation at essentially no additional cost and contributes to more efficient optimization.
\FloatBarrier

\section{Experiments and results}
\label{sec:experiments}
We present a comprehensive empirical study of optimization problems with embedded pruned NNs solved to global optimality. In addition to reporting outcomes in terms of computational performance and predictive accuracy, we investigate the underlying factors that drive these results.

\subsection{Design of the benchmarking study}
\label{sec:benchmark_design}
We benchmark pruning across five dimensions: the case study, the network morphology, the network complexity, the pruning technique with different sparsity levels, and the MILP solver. In the following, we describe each choice and the question it addresses; the results on the two case studies are presented in Sections~\ref{sec:exp_test_func} and~\ref{sec:exp_flowsheet}.

\paragraph{Case studies}
The first case study is the two-dimensional \textit{peaks} function (Section~\ref{sec:exp_test_func}), used as a controlled benchmark. The target function is smooth, low-dimensional and known in closed form. The second is a surrogate of an autothermal reformer flowsheet (Section~\ref{sec:exp_flowsheet}), with two
inputs, thirteen outputs and an additional inequality constraint on the product composition, used to verify that the effects observed in the controlled setting persist in a realistic engineering problem.

\paragraph{Morphology and complexity}
For the peaks case study we evaluate three morphologies, namely \textit{shallow and wide}, \textit{medium-sized}, and \textit{deep and narrow}, with four complexity levels between $\sim$10k and $\sim$200k parameters (Table~\ref{tab:design_architectures}). Parameter count is held approximately constant across morphologies at each level, so that depth--width effects are separable from the effect of model size. The two factors answer different questions: morphology determines how the ReLU units are distributed across layers, and therefore how interval-arithmetic bounds propagate, whereas complexity determines the size of the resulting MILP. Including the smallest complexities also establishes whether pruning is useful only as a remedy for overparameterization, or whether it improves tractability for surrogates that are already appropriately sized for the target. Conversely, the larger networks test the pruning philosophy that a large network, subsequently pruned, can outperform a small one trained directly~\cite{Arora2018_OptimizationDeepNetworks, Pham2025_OptimizationoverTrained}. For the flowsheet case study we keep the four complexity levels but fix the morphology to \textit{medium-sized}, since the morphology comparison is already shown in the controlled setting.

\paragraph{Pruning techniques and sparsity}
Three techniques are applied to every architecture: iterative magnitude-based weight pruning (unstructured), iterative magnitude-based node pruning (structured), and node pruning supported by SPR (Section~\ref{sec:regularization}). The first two contrast unstructured with structured sparsity under an identical, simple criterion; SPR is included as an advanced method designed for subsequent embedding in a MILP, and therefore indicates how much of the benefit is attributable to sparsity as such rather than to the particular criterion used to induce it. Sparsity levels are $\{0.8, 0.9, 0.95, 0.99\}$ for weight pruning and $\{0.6, 0.7, 0.8, 0.9\}$ for the two structured techniques; the ranges differ because the two definitions are not equivalent, as a given node sparsity removes a larger fraction of weights (cf. Section~\ref{sec:iterative_pruning}). Throughout this section, ``final sparsity'' denotes the fraction of weights removed in weight pruning and the fraction of neurons removed in node pruning. The relative pruning rate is fixed at $25\%$ per iteration in all cases, and all pruned networks undergo a fine-tuning step after the target sparsity is reached. SPR uses $\lambda = 10^{-4}$, $\alpha = 0.5$, $M = 1$. When magnitude-based techniques are in use, networks are trained with $L_2$ regularization ($\lambda = 5 \times 10^{-7}$); the selection of this setting is documented in Section~\ref{sec:regularization_analysis} and Appendix~\ref{app:regularization_analysis}. Unpruned networks of each architecture serve as the baseline.

\paragraph{Solvers}
Every trained model is solved with both the commercial solver Gurobi~\cite{GurobiOptimization2024_GurobiOptimizerReference} and the open-source solver HiGHS~\cite{Huangfu2017_Parallelizingdualrevised}. The two differ in, among others, their presolve, cut generation and branching heuristics, so agreement between them separates effects arising from properties of the formulation from effects specific to one solver implementation.

\paragraph{Protocol}
Each configuration is trained with 5 random initialization seeds, giving 780 trained models for the peaks case study and 260 for the flowsheet, each solved with both solvers with a time limit of 7{,}200~s.

\paragraph{Quantities analysed}
For each solve we record the optimization runtime and the predictive accuracy of the surrogate (on an independent test set), which together give the accuracy-runtime trade-off. 
The regression accuracy is calculated using the Mean Absolute Percentage Error (MAPE) between the target and the predictions of the neural network. The MAPE is defined as:
\begin{equation}
\mathrm{MAPE}[\%]
= \frac{100}{n} \sum_{i=1}^{n}
\frac{\left|y_i - \hat{y}_i\right|}{\max\left(|y_i|, \varepsilon\right)},
\end{equation}
where $n$ is the number of samples of the independent test set and $\varepsilon$ is a small constant to avoid numerical overflow.
To explain the mechanism behind that trade-off, we additionally record the three formulation properties identified in Section~\ref{sec:formulation_pruned}: the problem size before and after presolve, the number of integer variables in the resulting formulation, and the interval width of the big-$M$ bounds, together with the root LP relaxation gap and simplex throughput that these properties are expected to govern. Training time is recorded so that the one-off cost of pruning can be weighed against the recurring saving in optimization time.

\begin{table}[h!]
\centering
\caption{Architectures used in the benchmarking study, as depth $\times$ width. Three morphologies are crossed with four complexity levels for the peaks case study; the flowsheet case study uses the medium-sized column only.}
\label{tab:design_architectures}
\begin{tabular}{lccc}
\toprule
\# Parameters & Shallow \& Wide & Medium-Sized & Deep \& Narrow \\
\midrule
$\sim$10k   & $2 \times 93$  & $3 \times 66$  & $7 \times 38$ \\
$\sim$50k  & $2 \times 214$ & $4 \times 126$ & $13 \times 64$ \\
$\sim$100k & $2 \times 312$ & $5 \times 156$ & $16 \times 81$ \\
$\sim$200k & $2 \times 455$ & $6 \times 199$ & $20 \times 102$ \\
\bottomrule
\end{tabular}
\end{table}

\subsection{Illustrative test function}
\label{sec:exp_test_func}
We study a variety of pruning settings using a test function, namely the \textit{peaks} function ($f_{\text{peaks}}$). The peaks function is used as an illustrative benchmark that allows us to analyze the effect of pruning in a controlled setting. We do not claim that training a surrogate for this standalone low-dimensional function is, by itself, the most natural application of surrogate-based optimization; rather, the example serves to isolate the computational effects of pruning before moving to a more realistic engineering case study. The peaks function is nonlinear, with multiple local minima, and is often used to test global optimization. The function can be expressed by $f_{\text{peaks}}:\mathbb{R}^2 \rightarrow \mathbb{R}$, with:
\begin{equation}
    \label{eq:peaks_function}
f_{\text{peaks}}(x_1, x_2)
= 3(1 - x_1)^2 e^{-x_1^{2} - (x_2 + 1)^{2}}
\;-\; 10\left(\frac{x_1}{5} - x_1^{3} - x_2^{5}\right) e^{-x_1^{2} - x_2^{2}}
\;-\; \frac{1}{3} e^{-(x_1 + 1)^{2} - x_2^{2}}.
\end{equation}

We consider the peaks function over the domain $\mathcal{A}:=[-3,3]^2$, where it exhibits multiple local minima. The global minimum occurs at $(0.228, -1.626)$ with value $-6.551$. We train fully-connected NNs with ReLU activation to learn the peaks function on $\mathcal{A}$. We generate 5000 data points for training and validation and 500 data points for testing, using Latin hypercube sampling. The 5000 data points are split in 80\%/20\% for training and validation. The used training hyperparameters are listed in Appendix~\ref{app:trainig_parameters}.\\
We finally aim to solve the following optimization problem to global optimality by replacing the true peaks function $f_{\text{peaks}}$ with the trained neural network surrogate model $h_{\theta}^{\text{peaks}}$:

\begin{equation}
    \label{eq:peaks_optimization_problem}
    \begin{aligned}
    &\min_{x_1, x_2 \in \mathcal{A}} y
    &\text{s.t.} \quad y = h_{\theta}^{\text{peaks}}(x_1, x_2)
    \end{aligned}
\end{equation}

\subsubsection{Preliminary analysis on regularization}
\label{sec:regularization_analysis}
The regularizer shapes the weight-magnitude distribution that magnitude-based pruning exploits. A relevant criterion for selecting it is the position of each family on the accuracy-runtime frontier traced by sparsity.
As a preliminary experiment, we performed a grid search by training peaks surrogates of $\sim$50k parameters for three morphologies without regularization and with $L_1$ and $L_2$ penalties at $5\times10^{-7}$, $10^{-4}$ and $10^{-2}$. Each network was pruned iteratively by weight pruning and by node pruning, and every resulting model was solved to global optimality using Gurobi.\\
Figure~\ref{fig:pareto_l1_l2} shows the resulting trade-off. The $L_2$ frontier outperforms or matches the $L_1$ frontier in all three morphologies and for both pruning strategies. At $10^{-4}$, $L_2$ shifts the frontier toward shorter runtimes at essentially no accuracy cost (median MAPE increase $0.50\%$ under weight pruning, Table~\ref{tab:reg_strength}), whereas $L_1$ at the same strength costs an order of magnitude more ($8.39\%$). At $10^{-2}$ both families reach the shortest runtimes but with MAPE increases around $70\%$, i.e. complete model degradation.\\
The sparsity-inducing feature of $L_1$ therefore does not translate into a uniform advantage in this setting. We therefore adopted $L_2$ with $\lambda = 5\times10^{-7}$ for the main study because it perturbs the accuracy of trained network least while remaining on or near the frontier (see Table~\ref{tab:reg_strength}). While $\lambda=10^{-4}$ populates most of the $L_2$ frontier through shorter runtimes, this choice deliberately prioritizes accuracy. The same holds on the engineering case study (see Appendix~\ref{app:regularization_analysis}, Table~\ref{tab:reg_strength_reformer}). Nevertheless, regularization is not the central focus of this work, and this choice should be read as problem-dependent hyperparameter tuning rather than a general recommendation.

\begin{figure}[h!]
    \centering
    \includegraphics[width=\linewidth]{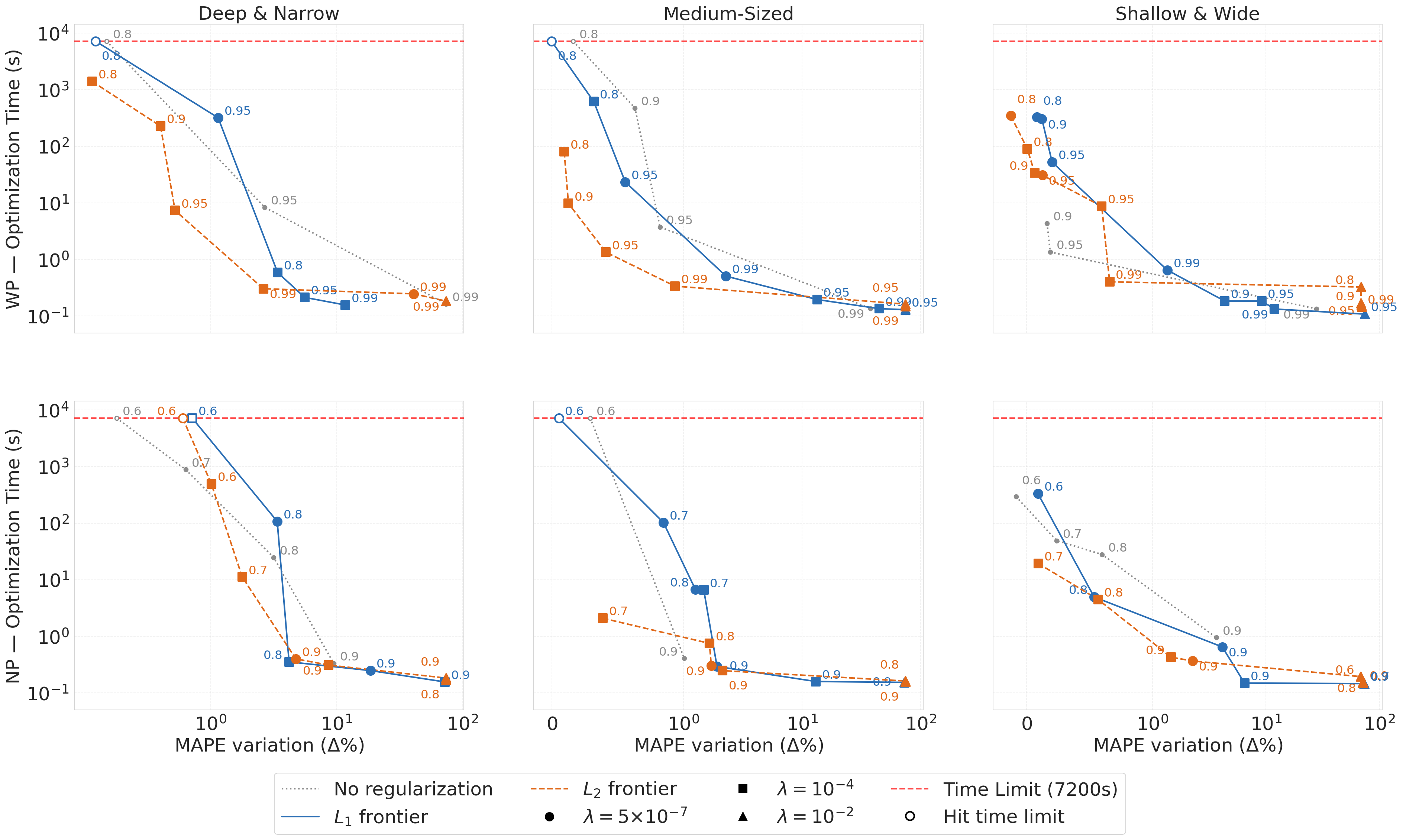}
    \caption{Accuracy-runtime trade-off for peaks surrogates ($\sim$50k parameters) trained without regularization (dotted) and with $L_1$ or $L_2$ penalties at three strengths, under weight pruning (WP, top) and node pruning (NP, bottom). Marker shapes denote the regularization strength $\lambda$, labels the sparsity level, and hollow markers runs that hit the time limit. The $L_2$ frontier dominates or matches the $L_1$ frontier across all architectures and both pruning strategies.}
    \label{fig:pareto_l1_l2}
\end{figure}

\begin{table}[htbp]
  \centering
  \caption{Accuracy cost of each regularization setting on the peaks case study, reported as MAPE variation (\%) against the unpruned baseline of the same architecture. Each entry pools every pruned configuration of that setting across architectures and sparsity levels.}
  \label{tab:reg_strength}
  \begin{tabular}{lrrrr}
    \toprule
    Regularizer & \multicolumn{2}{c}{WP} & \multicolumn{2}{c}{NP} \\
    \cmidrule(lr){2-3} \cmidrule(lr){4-5}
     & Median & Mean & Median & Mean \\
    \midrule
    None & 0.73 & 12.21 & 0.91 & 1.81 \\
    \midrule
    $L_1$ $5\times 10^{-7}$ & \textbf{0.39} & 6.74 & 0.97 & 2.77 \\
    $L_1$ $10^{-4}$ & 8.39 & 12.35 & 5.33 & 16.78 \\
    $L_1$ $10^{-2}$ & 74.21 & 74.21 & 73.84 & 73.66 \\
    \midrule
    $L_2$ $5\times 10^{-7}$ & 0.52 & 4.19 & \textbf{0.85} & \textbf{1.40} \\
    $L_2$ $10^{-4}$ & 0.50 & \textbf{0.59} & 1.23 & 2.06 \\
    $L_2$ $10^{-2}$ & 74.21 & 72.52 & 73.89 & 67.05 \\
    \bottomrule
  \end{tabular}
\end{table}

\subsubsection{Impact of pruning on regression accuracy}
\label{sec:exp-accuracy}
Figure~\ref{fig:mape_change_pruning} reports the MAPE variation of pruned networks relative to the unpruned baseline of the same architecture. 
Baseline accuracy (in the inset table) improves monotonically with parameter count for all three morphologies, from $\sim$0.9--1.9\% at 10k parameters to $\sim$0.3--0.4\% at 200k. Degradation under pruning grows with sparsity and shrinks with network complexity in every technique. The deep and narrow architecture is the most sensitive to pruning throughout, especially at low complexity under weight pruning, where MAPE increases by up to 70 percentage points at 99\% sparsity for the 10k network; no pruning technique is consistently preferable for this morphology. The shallow and wide architecture is the least sensitive under both weight and node pruning. Node pruning with SPR is comparatively most effective for the medium-sized architecture. At low sparsity, weight pruning slightly reduces MAPE below the baseline for the smallest (10k) networks, with a similar but smaller effect visible at 50k and 100k parameters.\\
The absence of MAPE increase at the largest parameter count, despite substantial overparameterization relative to the simple target, indicates that overfitting does not occur under this training procedure. The higher sensitivity of the deep and narrow architecture can be attributed to its few units per layer, which increases the probability that pruning cuts an entire path of information flow between input and output. The accuracy improvement observed for small networks under mild weight pruning is consistent with a regularizing effect.
\begin{figure}[h!]
    \centering
    \includegraphics[width=0.9\linewidth]{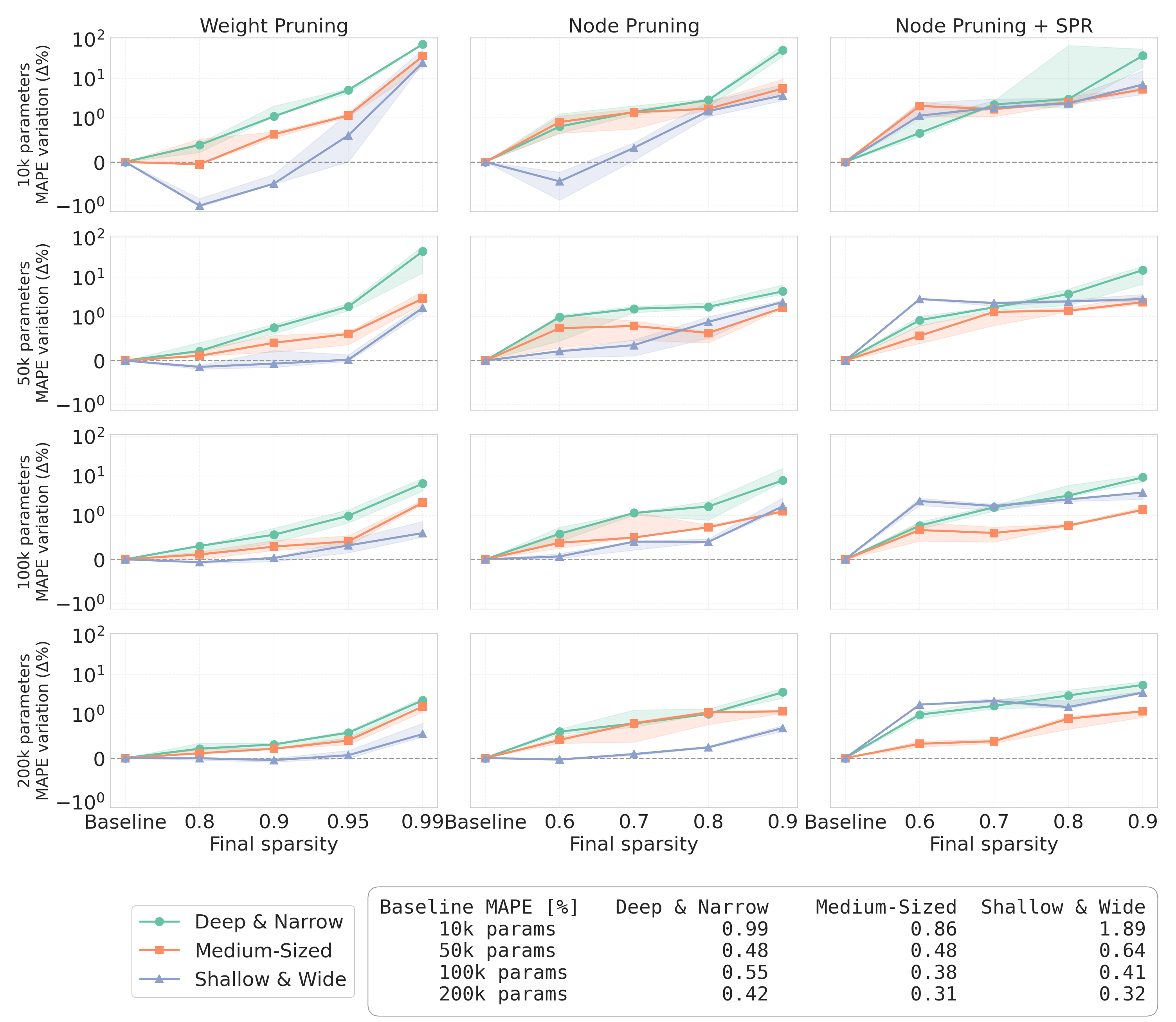}
    \caption{Comparison of the final regression accuracy measured relative to the unpruned baseline across different pruning techniques and model complexities.}
    \label{fig:mape_change_pruning}
\end{figure}

\subsubsection{Impact of pruning on the computational performance}
\label{sec:exp-comp-perf}
We measure the CPU time required to solve the optimization problems with Gurobi under varying pruning levels, architectures and network complexities. Figure~\ref{fig:speed_total} reports, for every combination, the optimization time and the combined training and pruning time against the unpruned baseline.
At 10k parameters, all three morphologies already converge within the 7,200-seconds time limit at very low sparsity, and pruning reduces optimization time from one to three orders of magnitude. From 50k parameters upward, dense networks increasingly time out, and the sparsity needed to recover tractability grows with complexity. Among the three techniques, node pruning with SPR restores tractability at the lowest sparsity, with  the largest (200k) networks solving already at $0.6$ sparsity, against $0.8$ for node pruning and $0.99$ for weight pruning. The total offline cost, i.e. training including pruning and fine-tuning, is comparable across morphologies at a given parameter count and increases with sparsity. Across architectures, the shallow and wide morphology consistently achieves the lowest optimization times, both as a baseline and under pruning. The deep and narrow morphology shows the opposite pattern at low sparsity, with substantially higher optimization times than the other two morphologies, but converges toward comparable times as sparsity increases.\\
The convergence speed of the smallest networks shows that pruning accelerates optimization even if the baseline is not very overparametrized. SPR's stronger effect at equal sparsity shows that a regularized weight distribution before pruning confirms to benefit the resulting formulation. Since the combined training and pruning cost is comparable across morphologies and complexities, the total computational cost is mainly governed by the optimization stage. Moreover, there is no significant difference in the offline cost of different pruning techniques. As sparsity increases, the optimization times of deep and narrow architectures approach those of shallow and wide ones, suggesting that the two morphologies reach a similar formulation complexity once sufficiently pruned. This is in line with the stronger bound-tightening effect of pruning in deep networks discussed in Section~\ref{sec:bound_tightening_pruning}.

\begin{figure}[h!]
    \centering
    \includegraphics[width=0.9\linewidth]{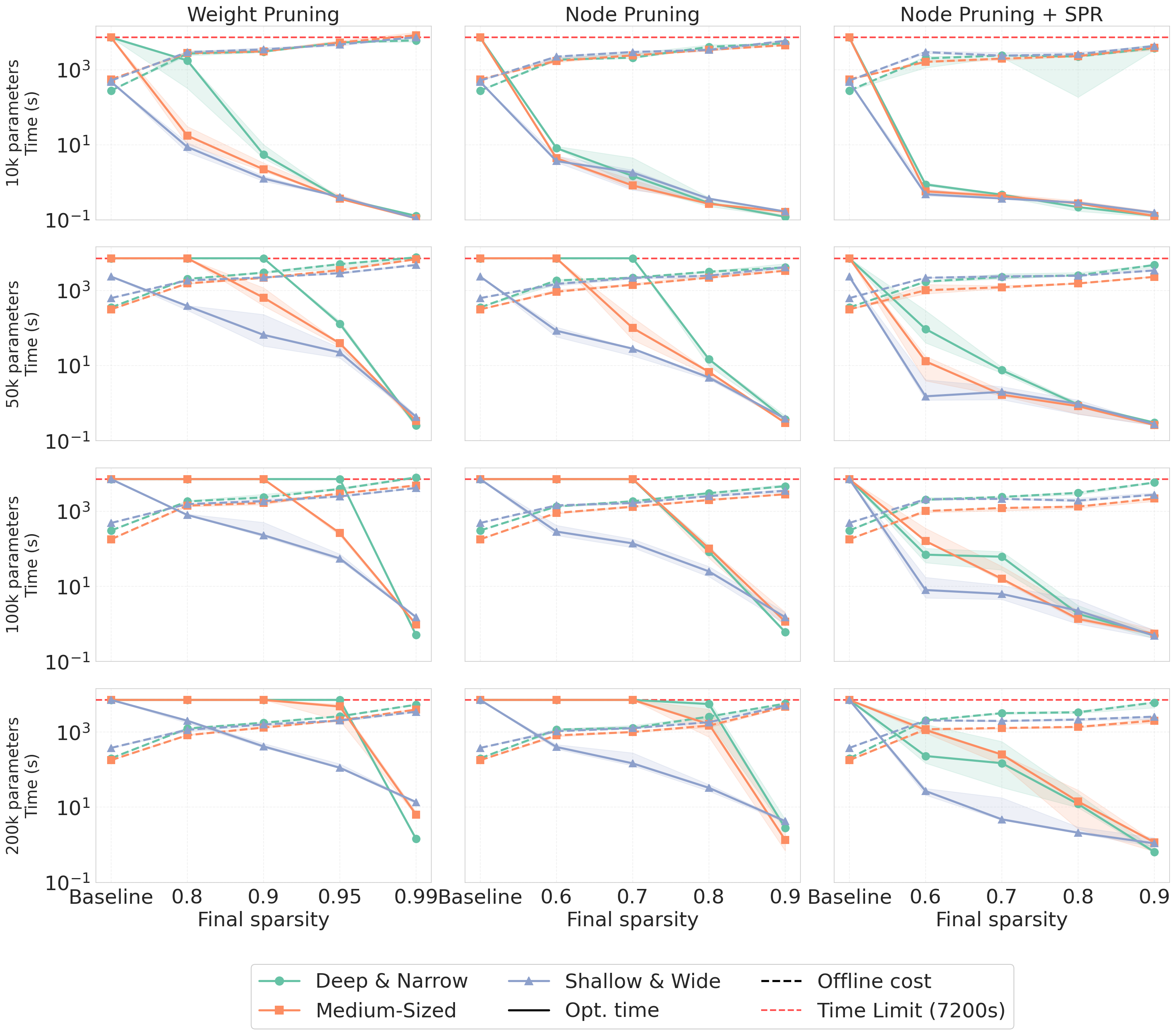}
    \caption{Total time decomposition (training vs.\ optimization) for each architecture and pruning strategy.}
    \label{fig:speed_total}
\end{figure}

Figure~\ref{fig:speed_mape} reports the accuracy-runtime trade-off as Pareto fronts, for the optimization time (left) and for the total time comprising offline cost and optimization (right), across all complexity levels and morphologies. The non-dominated configurations on optimization time are almost exclusively SPR-based node pruning and weight pruning, with few exceptions from magnitude-based node pruning. The shape of the fronts depends strongly on the baseline parametrization: the 10k networks reach a plateau in optimization time after mild sparsity, so that further pruning only degrades accuracy, whereas the 100k and 200k networks continue to gain up to three orders of magnitude at a moderate accuracy cost over most of the front. This suggests that a sufficiently pruned large network can attain comparable error and shorter optimization time with respect to a smaller unpruned one. Across morphologies, the shallow and wide front dominates the others at every complexity level, followed by the medium-sized one, while the deep and narrow front is displaced toward higher error.
On total time, the shallow and wide morphology is preferable in nearly all cases, and the non-dominated configurations concentrate at mild sparsity rather than at the aggressive levels that minimize optimization time. At higher accuracy regimes, the unpruned network is itself non-dominated.\\
The two panels correspond to different deployment settings. When the surrogate is trained once and reused across repeated optimization runs, optimization time dominates and aggressive pruning of a large network can be the appropriate choice (left). When the surrogate is trained for a single optimization instance, the offline cost is never amortized by a faster solve, so mild sparsity or baseline network can be reasonable (right).
The concentration of non-dominated points on SPR and weight pruning should be read as an indication of preferred techniques rather than a general ranking, and the appropriate strategy should still be assessed per case. The accuracy drop of the deep and narrow morphology is consistent with its higher sensitivity to pruning reported in Section~\ref{sec:exp-accuracy}. 
To assess solver dependence, we repeated the analysis with the open-source solver HiGHS (Appendix~\ref{app:highs_pareto}). The trends are unchanged across complexity levels and morphologies; the solvers diverge mainly at the largest complexities, where Gurobi is consistently faster.

\begin{figure}[h!]
    \centering
    \includegraphics[width=0.9\linewidth]{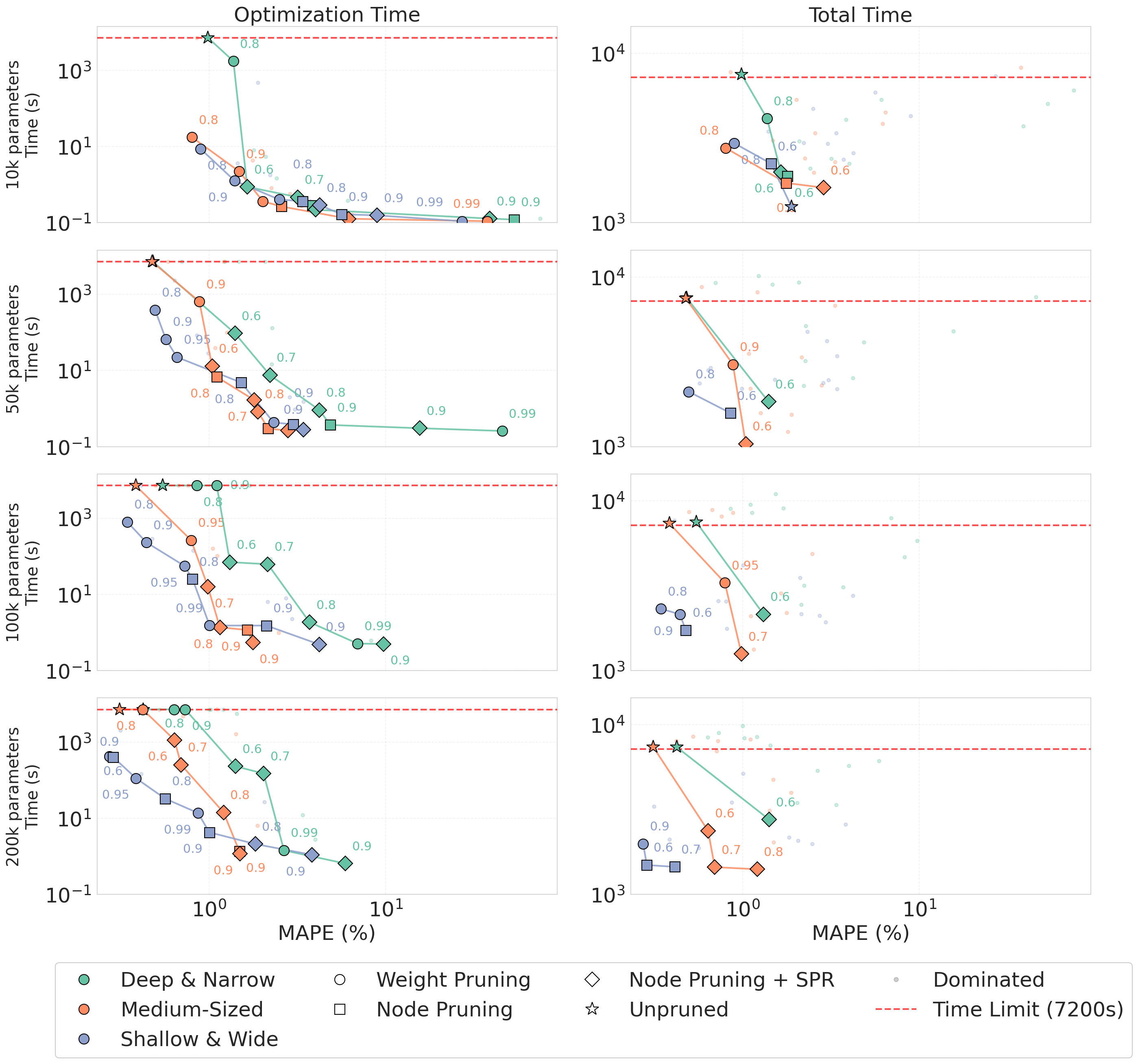}
    \caption{Accuracy-runtime Pareto fronts across network complexity and morphology, for optimization time (left) and total time including offline and optimization cost (right), solved with Gurobi. Each marker is one (technique, sparsity) configuration. A configuration is non-dominated only if no other configuration of the same morphology is at least as good on both axes and strictly better on one. The unpruned baseline is shown only where it is itself non-dominated.}
    \label{fig:speed_mape}
\end{figure}

\subsubsection{Analysis of the optimization formulation}
\label{sec:analysis_experiments_formulation}
As discussed in Section~\ref{sec:problem_formulation_sparsity}, we argue that the mechanistic explanation resides in key properties of the big-M formulation of ReLU networks, including (i) problem size, (ii) number of integer variables, and (iii) tightness of the big-M bounds. In the following, we present a comprehensive empirical analysis of these factors to support our hypothesis.

\paragraph{Problem size and number of integer variables}
Pruning reduces the size of the algebraic formulation. Masked weights enter the model as zero coefficients, and when entire neurons are pruned the vanished units can either be kept in the weight matrices as zeroed rows and columns or deleted, yielding a smaller dense network. The two routes are equivalent in practice and the solver's presolve reduces both to the same MILP. We show this in Appendix~\ref{app:formulation_equivalence}.
Table~\ref{tab:problem_size_reduction_summary} reports the resulting reduction in presolved problem size, relative to the presolved unpruned model of the same architecture, for the smallest (10k) and largest (200k) complexity of each morphology.
The reduction in variables, constraints and integer variables grows with sparsity for all three techniques. At matched architecture and sparsity, node pruning with SPR generally achieves the strongest reduction in nearly every case, by 3--10 percentage points over magnitude-based node pruning at the same sparsity. For node pruning and SPR, the reduction at a given sparsity is essentially unchanged between the 10k and 200k complexity of the same morphology. Weight pruning does not share this property. Its reduction diverges between complexities, and consistently more so for the deep and narrow morphology than for the shallow and wide one.\\
Variables and constraints elimination is found to be an inexpensive step for both Gurobi and HiGHS (avg. presolve time $<1$~s). Presolve removes most of the excess of the masked encoding, but only marginally reduces integer variables, independently of sparsity. The sparsity-dependent reduction in integer variables, which governs the B\&B search, is therefore due to pruning itself rather than to solver-side simplification (cf. Appendix~\ref{app:formulation_equivalence}).\\
The reduction in integer variables (cf. Eq.~\eqref{eq:bigMformulation}) directly shrinks the combinatorial space explored by the B\&B algorithm, as discussed in Section~\ref{sec:problem_formulation_sparsity}. Node pruning removes entire structures deterministically, so its reduction in problem size is governed only by sparsity. In weight pruning, by contrast, a neuron whose incoming weights happen to be fully pruned becomes dead and is eliminated along with it, as discussed in Section~\ref{sec:problem_formulation_sparsity}. This implicit node removal is more likely in narrow layers, where fewer incoming connections must be pruned to isolate a neuron entirely, which explains the stronger reduction observed for the deep. SPR's larger reduction relative to node pruning at equal sparsity can be explained by its regularization term shaping the weight distribution toward structured sparsity during training, which likely also contributes to the strong reduction in optimization time reported in Section~\ref{sec:exp-comp-perf}.\\
Beyond combinatorial complexity, the size of the LP itself also affects solver performance. Figure~\ref{fig:solver_simplex_it_root_LPgap} (top) shows the average number of simplex iterations per second as a function of sparsity, for every complexity level. From 50k parameters upward, throughput increases with sparsity for all architectures and techniques, since smaller LPs enable faster matrix factorization and inversion in the simplex method. At 10k parameters this trend reverses. However, this coincides with total solve times below one second, and the reported rate is computed over few iterations only. Node pruning and SPR reach higher throughput than weight pruning at comparable sparsity.

\begin{table}[htbp]
  \centering
  \setlength{\tabcolsep}{2.8pt}
  \caption{Reduction (\%) in presolved problem size, relative to the presolved unpruned model of the same architecture. Each shape is reported at its largest (200k) and smallest (10k) parameter budget. Entries are means over five seeds, rounded to integers. Median standard deviation by technique, worst case in brackets: Weight Pruning 2.1 (4.8); Node Pruning 0.2 (3.7); Node Pruning + SPR 1.5 (3.8). Comparing the two columns of each shape: under node pruning the reduction tracks the sparsity irrespective of model size, while under weight pruning it does not, because whether a neuron becomes fully isolated depends on layer width.}
  \label{tab:problem_size_reduction_summary}
  \begin{tabular}{lllcccccc}
    \toprule
    \multicolumn{3}{c}{} & \multicolumn{2}{c}{Deep \& Narrow} & \multicolumn{2}{c}{Medium-Sized} & \multicolumn{2}{c}{Shallow \& Wide} \\
    \cmidrule(lr){4-5} \cmidrule(lr){6-7} \cmidrule(lr){8-9}
    Technique & Sparsity & Quantity & 200k & 10k & 200k & 10k & 200k & 10k \\
    \midrule
    Weight Pruning & 0.8 & \# Vars & $32$ & $29$ & $28$ & $22$ & $16$ & $20$ \\
     &  & \# Cons & $32$ & $29$ & $28$ & $21$ & $15$ & $19$ \\
     &  & \# Int. Vars & $32$ & $29$ & $27$ & $21$ & $13$ & $17$ \\
     & 0.9 & \# Vars & $42$ & $47$ & $35$ & $38$ & $28$ & $40$ \\
     &  & \# Cons & $42$ & $46$ & $35$ & $37$ & $27$ & $38$ \\
     &  & \# Int. Vars & $42$ & $45$ & $35$ & $35$ & $24$ & $35$ \\
     & 0.95 & \# Vars & $54$ & $71$ & $47$ & $57$ & $44$ & $55$ \\
     &  & \# Cons & $54$ & $71$ & $47$ & $55$ & $42$ & $53$ \\
     &  & \# Int. Vars & $55$ & $68$ & $48$ & $51$ & $38$ & $49$ \\
     & 0.99 & \# Vars & $90$ & $99$ & $78$ & $91$ & $66$ & $89$ \\
     &  & \# Cons & $90$ & $99$ & $77$ & $90$ & $65$ & $88$ \\
     &  & \# Int. Vars & $91$ & $98$ & $76$ & $88$ & $61$ & $86$ \\
    \midrule
    Node Pruning & 0.6 & \# Vars & $60$ & $60$ & $60$ & $59$ & $59$ & $58$ \\
     &  & \# Cons & $60$ & $60$ & $60$ & $59$ & $58$ & $58$ \\
     &  & \# Int. Vars & $60$ & $61$ & $59$ & $58$ & $56$ & $56$ \\
     & 0.7 & \# Vars & $70$ & $68$ & $70$ & $69$ & $69$ & $69$ \\
     &  & \# Cons & $70$ & $69$ & $70$ & $69$ & $69$ & $69$ \\
     &  & \# Int. Vars & $70$ & $69$ & $69$ & $68$ & $67$ & $68$ \\
     & 0.8 & \# Vars & $79$ & $79$ & $80$ & $80$ & $79$ & $80$ \\
     &  & \# Cons & $79$ & $80$ & $80$ & $80$ & $79$ & $80$ \\
     &  & \# Int. Vars & $79$ & $80$ & $79$ & $79$ & $78$ & $80$ \\
     & 0.9 & \# Vars & $91$ & $92$ & $90$ & $89$ & $89$ & $90$ \\
     &  & \# Cons & $91$ & $93$ & $90$ & $89$ & $89$ & $90$ \\
     &  & \# Int. Vars & $92$ & $93$ & $90$ & $89$ & $89$ & $90$ \\
    \midrule
    Node Pruning + SPR & 0.6 & \# Vars & $69$ & $64$ & $67$ & $64$ & $74$ & $66$ \\
     &  & \# Cons & $70$ & $64$ & $67$ & $64$ & $74$ & $66$ \\
     &  & \# Int. Vars & $72$ & $64$ & $68$ & $63$ & $74$ & $64$ \\
     & 0.7 & \# Vars & $79$ & $73$ & $76$ & $71$ & $78$ & $73$ \\
     &  & \# Cons & $80$ & $73$ & $77$ & $71$ & $78$ & $73$ \\
     &  & \# Int. Vars & $83$ & $73$ & $78$ & $71$ & $77$ & $72$ \\
     & 0.8 & \# Vars & $84$ & $83$ & $85$ & $80$ & $85$ & $81$ \\
     &  & \# Cons & $84$ & $83$ & $85$ & $80$ & $84$ & $81$ \\
     &  & \# Int. Vars & $85$ & $83$ & $86$ & $80$ & $83$ & $81$ \\
     & 0.9 & \# Vars & $93$ & $92$ & $93$ & $90$ & $92$ & $90$ \\
     &  & \# Cons & $94$ & $92$ & $93$ & $90$ & $92$ & $90$ \\
     &  & \# Int. Vars & $95$ & $92$ & $93$ & $90$ & $91$ & $90$ \\
    \bottomrule
  \end{tabular}
\end{table}

\begin{figure}[h!]
    \centering
        \includegraphics[width=\linewidth]{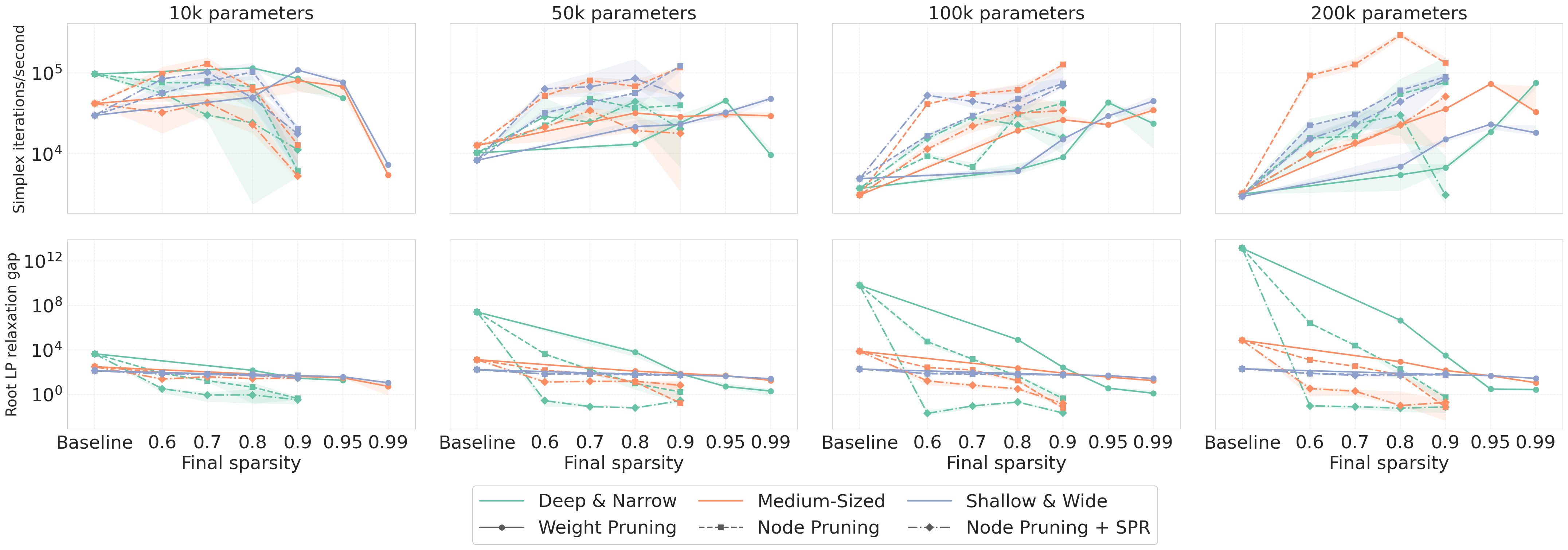}
        \caption{Impact of sparsity on Gurobi solver performance. Higher sparsity reduces the problem size, enabling generally faster simplex iterations (top), while tighter big-M bounds from pruning strengthen the LP relaxation, reducing the root node gap (bottom) and accelerating branch and bound convergence.}
    \label{fig:solver_simplex_it_root_LPgap}
\end{figure}

\paragraph{Bound tightening}
Figure~\ref{fig:bounds-individual} shows the average big-M width $U_i^{(j)} - L_i^{(j)}$ across all neurons as a function of sparsity, for every morphology, complexity level and pruning technique, and Figure~\ref{fig:bounds} reports the corresponding distributions across neurons for the smallest and largest network.
The bound width decreases monotonically with sparsity for all three techniques, all morphologies and all complexity levels, consistent with Proposition~\ref{prop:monotonic-tightening-short} and Proposition~\ref{prop:strict-tight-short} (cf. Section~\ref{sec:bound_tightening_pruning}) and despite the final fine-tuning step. The baseline width depends strongly on both depth and complexity. For instance, it reaches $\sim$$10^{10}$ for the largest deep and narrow architecture and decreases by several orders of magnitude at lower complexity for the same morphology, whereas for the shallow and wide architecture it remains tighter at every complexity level. Node pruning with SPR produces the steepest reduction, collapsing the width by up to ten orders of magnitude at the mildest sparsity considered ($s=0.6$) and remaining nearly flat thereafter;
weight and node pruning reduce it progressively over the sparsity range. The distributions in Figure~\ref{fig:bounds} show that the reduction affects the entire population of neurons. Figure~\ref{fig:solver_simplex_it_root_LPgap} (bottom) presents the root node LP relaxation gap, defined as the distance between the lower bound and the best integer solution found at the first iteration (i.e., how tight the first relaxation is). The gap follows the bound widths: it is largest for the unpruned deep and narrow architecture, reaching $\sim$$10^{13}$ at 200k parameters, and decreases by up to twelve orders of magnitude under pruning. The shallow and wide architecture maintains a lower gap at all sparsity levels and complexities, including the unpruned baseline.\\
The dependence of the interval width on the NN's depth is driven by the compounding effect of interval arithmetic through successive layers. Appendix~\ref{app:bounds-deep} reports the corresponding layer-wise analysis for the deep and most complex architecture. The fine-tuning step does not counteract this tightening, since it does not significantly change the magnitudes of the surviving parameters. The immediate collapse under SPR indicates that the structured regularization already produces near vanishing weight structures during training. This is consistent with SPR speed up at lower sparsity than the magnitude-based techniques in Section~\ref{sec:exp-comp-perf}. The consequent reduction in relaxation gap indicates a tighter lower bounds during the B\&B search, enabling earlier elimination of suboptimal branches. This mechanism underlines the optimization speedups reported in Figure~\ref{fig:speed_total}, in particular for the the deep and narrow morphology where a large benefit is observed.

\begin{figure}[h!]
    \centering
    \includegraphics[width=0.9\linewidth]{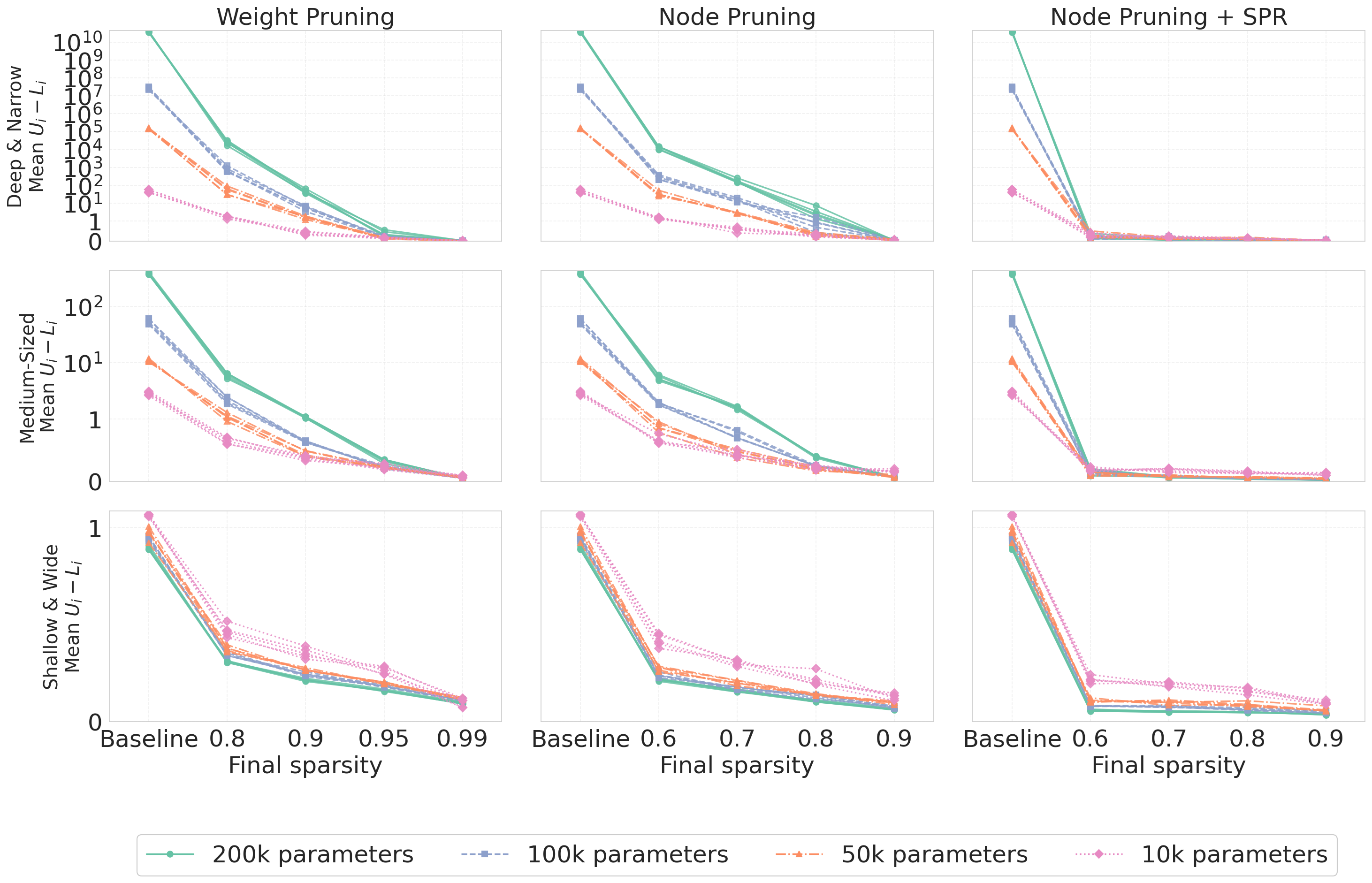}
    \caption{Monotonic decrease of average big-$M$ widths with increasing sparsity for individual training runs.}
    \label{fig:bounds-individual}
\end{figure}

\begin{figure}[h!]
    \centering
    \includegraphics[width=0.9\linewidth]{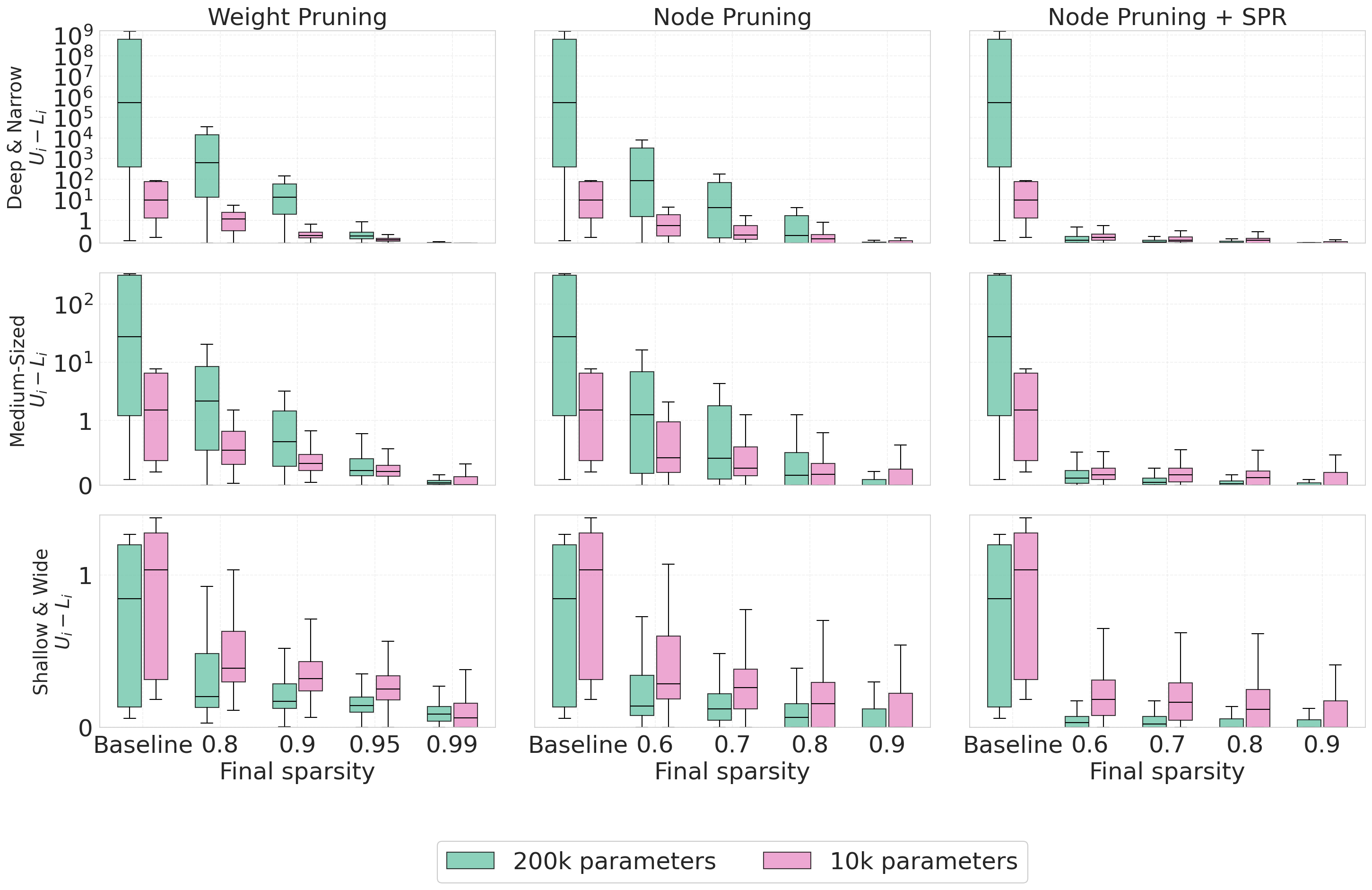}
    \caption{Distribution of interval width $U_i - L_i$ of big-M bounds across all neurons as a function of sparsity. Deep networks benefit most due to compounding effects of interval arithmetic bound propagation.}
    \label{fig:bounds}
\end{figure}

\subsection{Optimization of a chemical process flowsheet}
\label{sec:exp_flowsheet}
To validate the engineering applicability of our findings beyond the illustrative test function, we apply both pruning strategies to a realistic case study involving surrogate-based optimization of a chemical process flowsheet.

\subsubsection{Problem description}
We apply the two pruning methods to a neural network surrogate representing a chemical process flowsheet.
We use an autothermal reformer (ATR) flowsheet as a case study, as reported in the documentation of the OMLT library~\cite{OMLT_ATR_Example} and previously adapted by Bugosen et al.~\cite{Bugosen2024_ProcessFlowsheetOptimization}.
The main objective of the autothermal reformer is to produce syngas, a mixture primarily composed of H$_2$ and CO.
The process flowsheet is illustrated in Figure~\ref{fig:flowsheet}.
A mixture of natural gas, steam, and air is fed into the ATR reactor, where the reforming reactions occur.
The hot syngas exiting the reformer is circulated through a heat exchanger (feed-effluent heat exchanger) to preheat the incoming natural gas feed.
The preheated natural gas is then expanded to generate electrical power before being fed into the reactor, closing the loop.
A bypass valve regulates the fraction of natural gas fed into the reformer versus passed around it, while air is compressed and cooled in a two-stage compression process before entering the reformer.
The objective of the case study is to maximize the molar fraction of hydrogen (x$_{\text{H}_2}$) in the product stream. \\
We train a medium-sized neural network surrogate model representing the entire flowsheet, with two inputs corresponding to the degrees of freedom: (1) the bypass fraction ($\alpha$), regulating how much natural gas is fed into the reformer versus bypassed, and (2) the natural gas to steam ratio ($r$). 
The network has 13 outputs: the steam flow rate ($F_{\text{steam}}$), the reformer heat duty ($Q_{\text{ref}}$), and 11 outputs representing the molar fractions ($x_i$) of the product stream components (Ar, C$_2$H$_6$, C$_3$H$_8$, C$_4$H$_{10}$, CH$_4$, CO, CO$_2$, H$_2$, H$_2$O, N$_2$, O$_2$). 
The complete list of inputs and outputs are provided in Appendix~\ref{app:trainig_parameters} (Table~\ref{tab:nn_surrogate_inputs} and Table~\ref{tab:nn_surrogate_outputs}, respectively).\\
NNs are trained on the dataset provided by OMLT~\cite{OMLT_ATR_Example}, comprising 2,500 data points split 80\%/20\% into training and validation sets.
An additional independent test dataset with 300 data points is generated for the final evaluation of model accuracy on unseen data.
The trained neural network surrogate is embedded into an optimization problem to be solved deterministically to global optimality through Pyomo and OMLT, with the objective of maximizing the hydrogen content of the product stream.
The optimization problem is formulated as follows:
\begin{equation}
    \label{eq:reformer_opt}
    \begin{aligned}
        \max_{\alpha, r} \quad & x_{\text{H}_2} \\
        \text{s.t.} \quad & \mathbf{y} = h_{\mathbf{\theta}}(\alpha, r) \\
        & x_{\text{N}_2} \leq 0.34 \\
        & 0.1 \leq \alpha \leq 0.8 \\
        & 0.8 \leq r \leq 1.2,
    \end{aligned}
\end{equation}
where $\alpha$ is the bypass fraction, $r$ is the natural gas to steam ratio, $\mathbf{y}$ is the vector of neural network outputs, and $h_{\mathbf{\theta}}(\alpha, r)$ represents the neural network surrogate with trained parameters $\mathbf{\theta}$.
The neural network is formulated using the big-M formulation for ReLU networks (Eq.~\eqref{eq:bigMformulation}) as implemented in the OMLT library.

\begin{figure}[h!]
    \centering
    \includegraphics[width=0.8\linewidth]{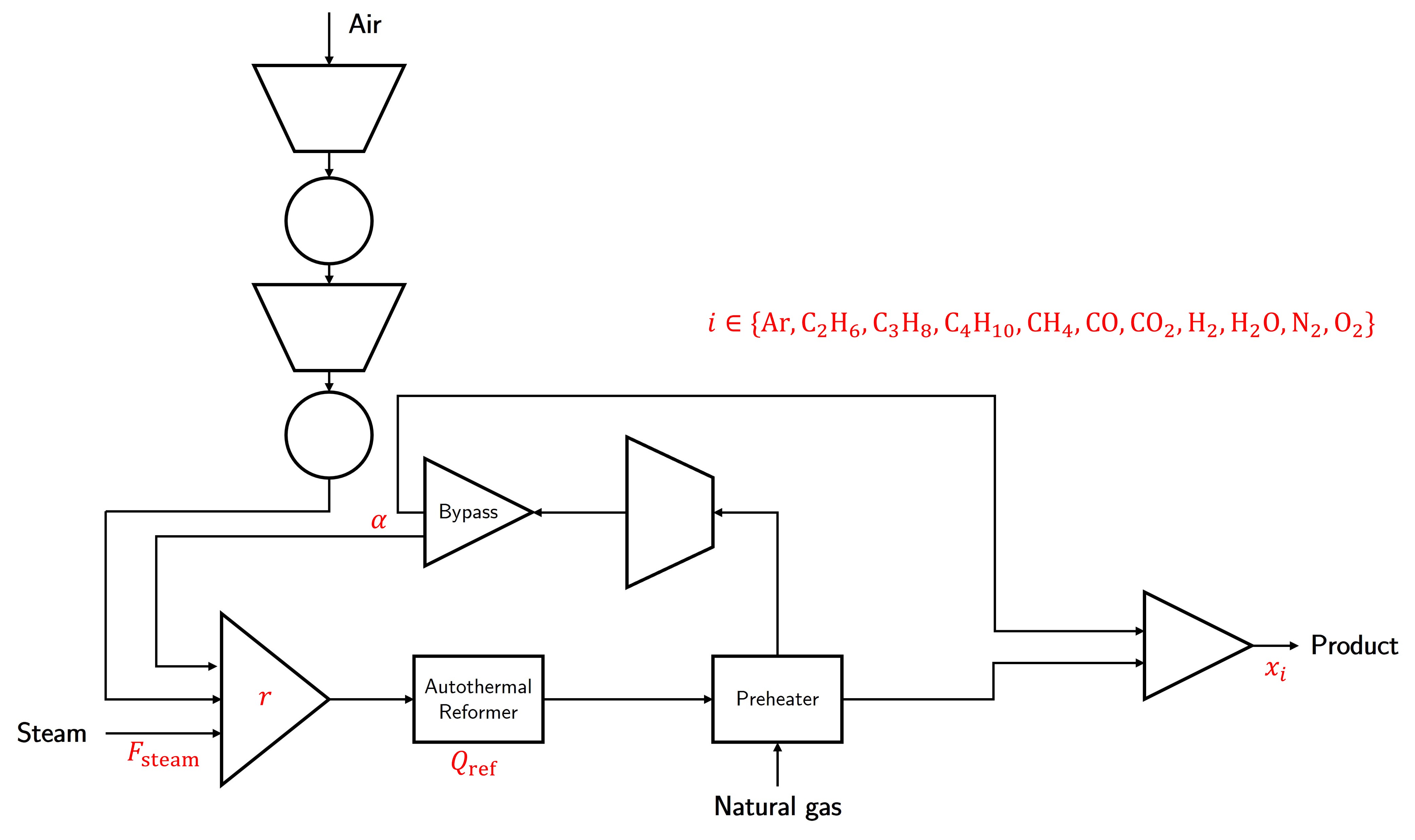}
    \caption{Process flow diagram of the autothermal reformer flowsheet. Natural gas, steam, and air are fed into the reformer reactor to produce syngas. The hot product stream preheats the natural gas feed through a heat exchanger before expansion for power generation.}
    \label{fig:flowsheet}
\end{figure}

\subsubsection{Optimization of pruned networks across varying complexities}
Figure~\ref{fig:sweep_gurobi} shows that pruning improves the accuracy-runtime trade-off for the ATR flowsheet optimization problem, with stronger effects as network complexity increases. For the smallest networks ($\sim 10$k parameters), pruning still provides roughly one order of magnitude speedup, although the unpruned formulation is already relatively tractable. For larger networks, unpruned or mildly pruned models frequently reach the 7,200-second time limit, whereas sufficiently pruned networks solve within seconds or less while maintaining MAPE values mostly below $1\%$. Notably, SPR-based node pruning provides a strong speedup for the $\sim 200$k network already at mild sparsity ($s=0.6$), while magnitude-based node pruning at the same sparsity still reaches the time limit; however, this improvement comes with some loss in prediction accuracy. Unlike the test function case study, where the nondominated configurations were distributed between weight pruning and SPR depending on morphology (Section~\ref{sec:exp-comp-perf}), here weight pruning occupies most of the Pareto frontier across complexity levels.\\
These results show that pruning is most useful when the embedded network is
large enough to make the big-$M$ formulation difficult to solve.
Small networks may already be sufficiently compact, so pruning provides smaller absolute gains, even if the relative speedup remains visible.
The different ranking of techniques between the two case studies indicates that, while pruning is beneficial in both, the preferred strategy and sparsity level depend not only on the network architecture but also on the case study itself, and should therefore be selected as part of the surrogate development rather than fixed a priori. The strong speed up of SPR also suggests that a more comprehensive regularization hyperparameter search could further improve pruning effectiveness, although this is outside the scope of the present work.

\begin{figure}[h!]
    \centering
    \begin{subfigure}{0.48\linewidth}
        \centering
        \includegraphics[width=\linewidth]{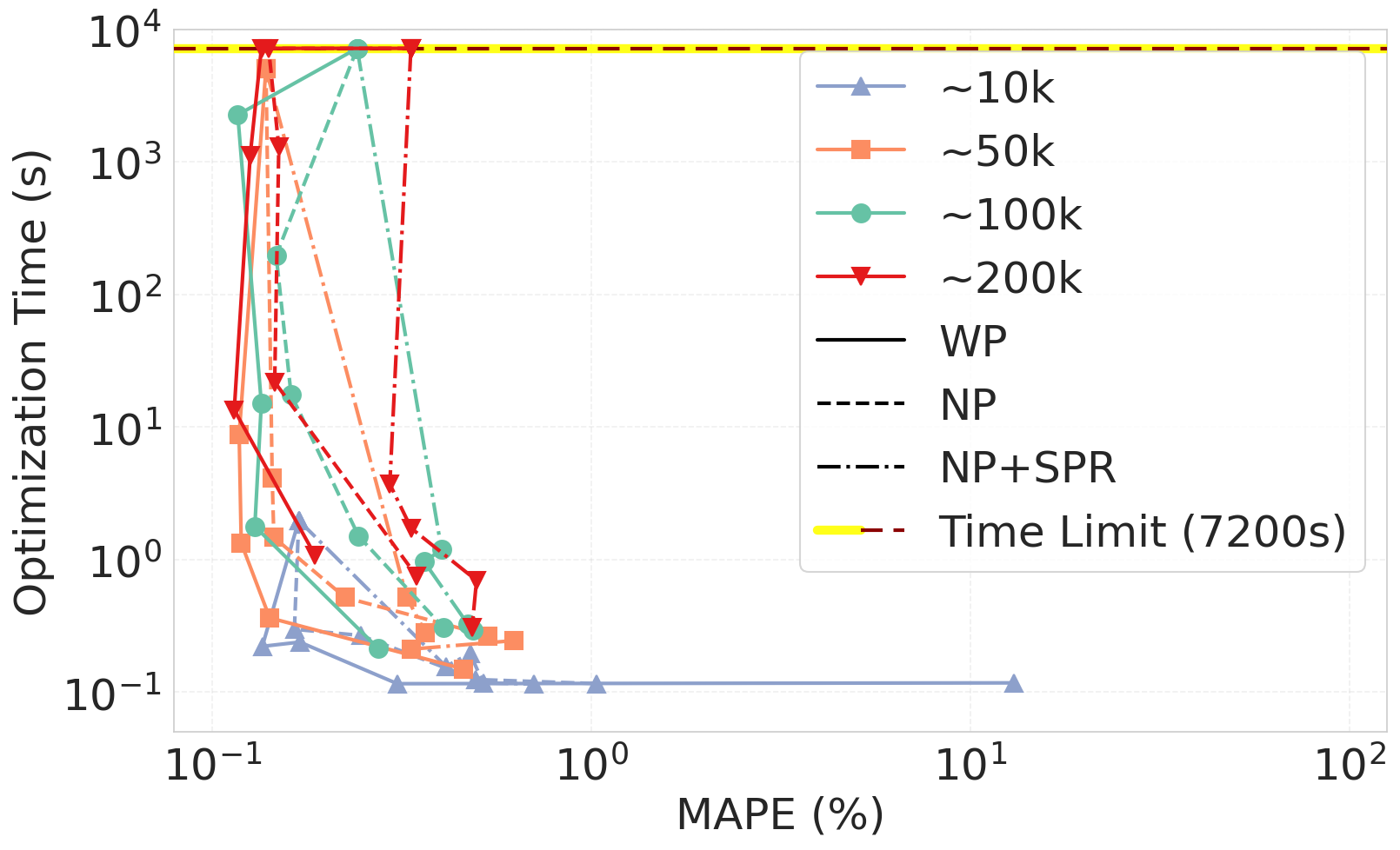}
        \caption{Gurobi}
        \label{fig:sweep_gurobi}
    \end{subfigure}
    \hfill
    \begin{subfigure}{0.48\linewidth}
        \centering
        \includegraphics[width=\linewidth]{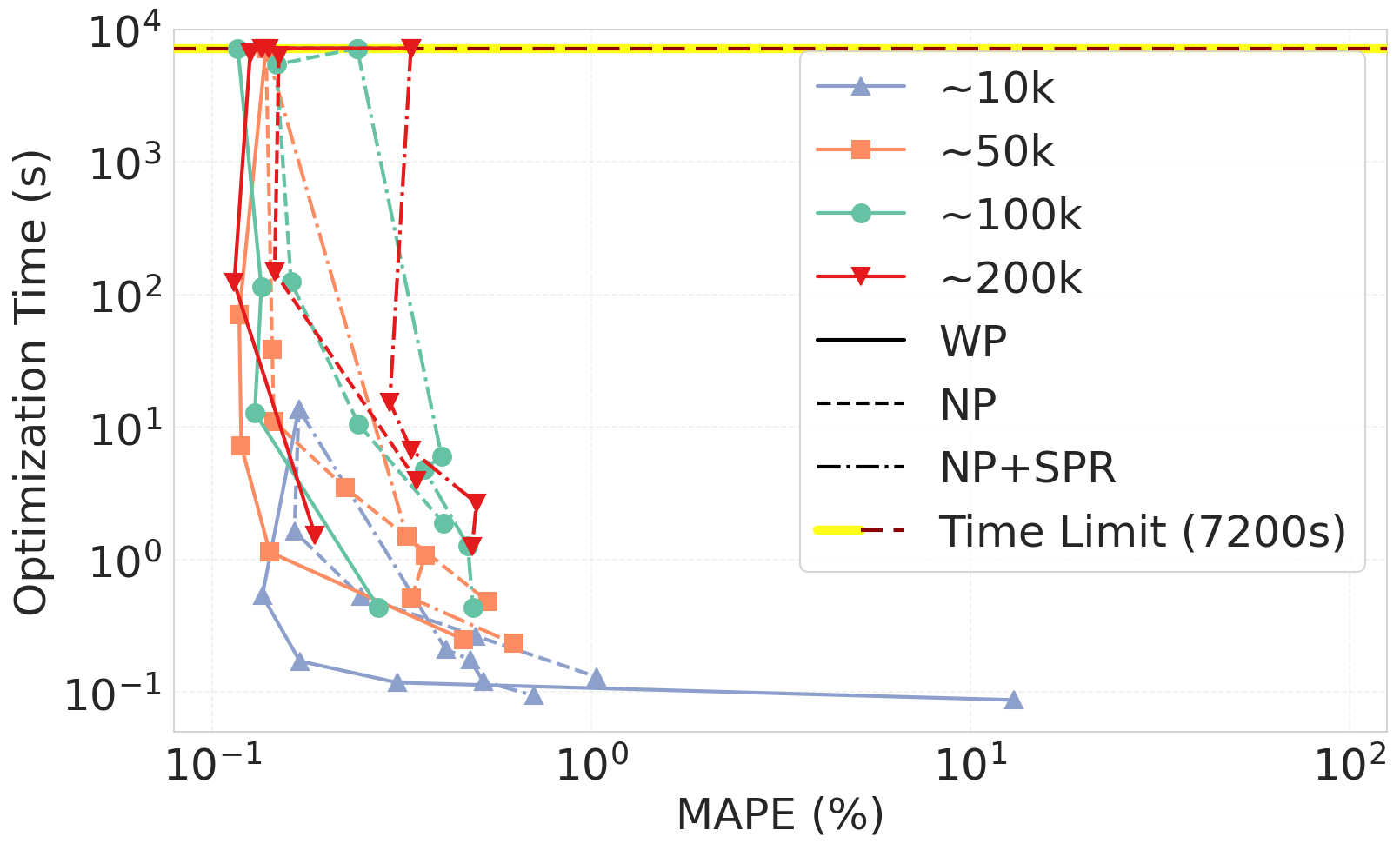}
        \caption{HiGHS}
        \label{fig:sweep_highs}
    \end{subfigure}
    \caption{Accuracy-runtime trade-off for the ATR flowsheet optimization problem across neural networks of increasing complexity. A qualitatively similar Pareto front is observed with both the commercial solver Gurobi and the open-source solver HiGHS.}
    \label{fig:sweep}
\end{figure}

\subsubsection{Bounds tightening over pruned networks with varying complexity}
Figure~\ref{fig:bounds_sweep} shows that pruning consistently reduces the IA-based preactivation bound widths across network sizes and pruning strategies, according to the results observed in Figure~\ref{fig:bounds} for magnitude-based unstructured and structured pruning. 
SPR-based node pruning produces particularly compressed bound-width distributions compared with the other pruning strategies, especially for the larger networks.\\
These results support the mechanism proposed in Section~\ref{sec:bound_tightening_pruning} and are consistent with the analysis in the illustrative case study~\ref{sec:exp_test_func}.

\begin{figure}[h!]
    \centering

    \begin{subfigure}{0.32\linewidth}
        \centering
        \includegraphics[width=\linewidth]{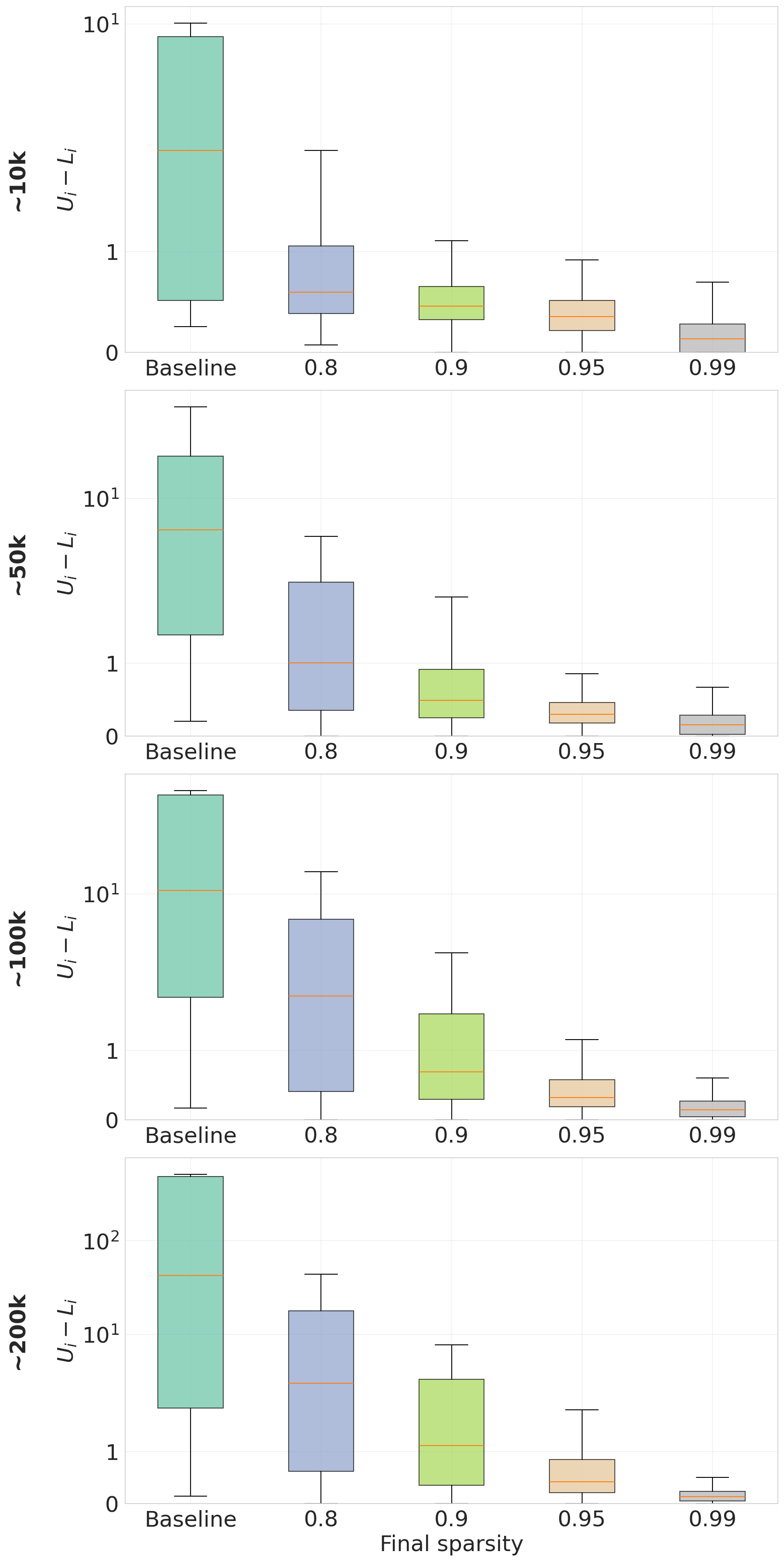}
        \caption{Weight pruning}
        \label{fig:bounds_sweep_IMBP}
    \end{subfigure}
    \hfill
    \begin{subfigure}{0.32\linewidth}
        \centering
        \includegraphics[width=\linewidth]{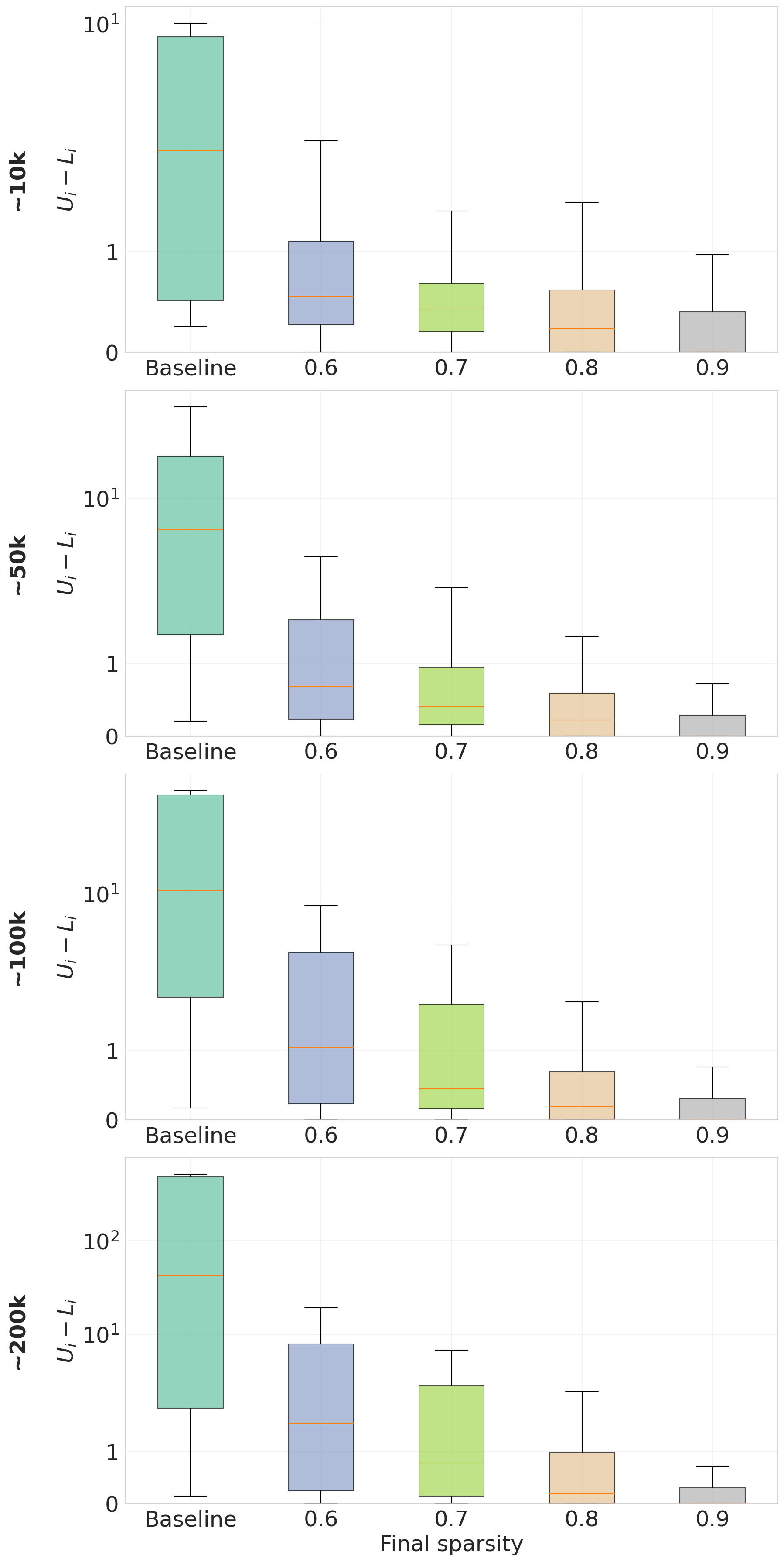}
        \caption{Node pruning}
        \label{fig:bounds_sweep_IMBP_ST}
    \end{subfigure}
    \hfill
    \begin{subfigure}{0.32\linewidth}
        \centering
        \includegraphics[width=\linewidth]{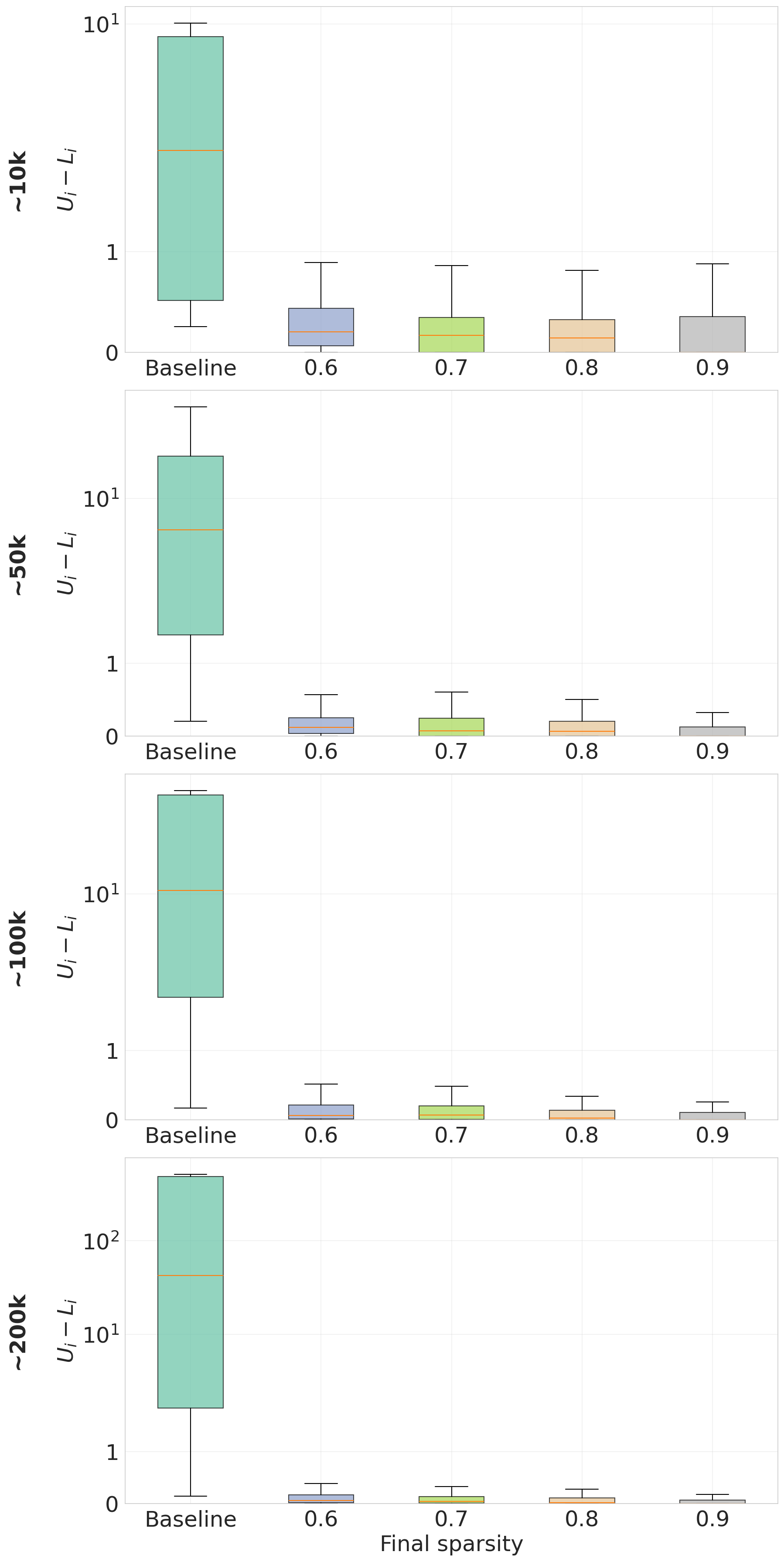}
        \caption{Node pruning with SPR}
        \label{fig:bounds_sweep_SPR}
    \end{subfigure}

    \caption{Distribution of IA-based preactivation bound widths for the ATR flowsheet surrogate across increasing model complexity and pruning levels. SPR-supported node pruning shows particularly tightened bound-width distributions.}
    \label{fig:bounds_sweep}
\end{figure}

\subsubsection{Optimization of pruned networks with an open-source solver}
\label{sec:open-source-highs}
Figure~\ref{fig:sweep_highs} shows that the main pruning trends also hold when the ATR flowsheet problems are solved with the open-source solver HiGHS~\cite{Huangfu2017_Parallelizingdualrevised}.
For the smallest networks, pruning provides up to roughly two orders of magnitude speedup. 
For larger networks, unpruned or mildly pruned formulations often reach the 2-hours time limit, while sufficiently pruned networks solve in seconds or less. 
As observed with Gurobi, SPR-based node pruning gives a strong speedup for the $\sim 200$k network already at mild sparsity (although less evident than when using Gurobi), whereas magnitude-based node pruning at the same sparsity still reaches the time limit.\\
These results indicate that the observed benefits are not specific to the used solver. 
Although solver implementation details affect absolute runtimes, pruning improves formulation properties that are broadly relevant for MILP solvers: smaller or sparser constraint matrices, fewer binary variables for structured pruning, and tighter big-$M$ bounds.

\FloatBarrier
\section{Conclusion}
Pruning can improve the computational efficiency of deterministic global optimization over trained ReLU surrogates, at an accuracy cost that depends on the network being pruned.
The main aims of this work were (i) to investigate how different pruning strategies impact the computational performance of deterministic global optimization across ReLU network morphologies and complexities, (ii) to understand the underlying mechanisms driving these improvements, and (iii) to validate the findings empirically, including on an illustrative testing function and a realistic engineering application.
On the illustrative case study, we showed that weight pruning, node pruning and node pruning supported by SPR all accelerate the solution of the mixed-integer formulation, with speedups spanning one to four orders of magnitude depending on the complexity of the original network. Weight pruning and SPR occupy most of the Pareto frontier of the accuracy-runtime trade-off. SPR induces faster convergence to the optimum at milder sparsity for medium-to-big networks, whereas weight pruning is particularly effective for the smaller networks. Deep and narrow architectures benefit the most in runtime, but degrade in accuracy under every technique. The computational gains were attributed to three factors: (i) reduction in problem size, (ii) decrease in the number of integer variables, and (iii) tightening of big-M bounds. These effects collectively strengthen the LP relaxation and decrease the combinatorial complexity of the MILP. We further verified that the trends hold with both a commercial (Gurobi) and an open-source (HiGHS) MILP solver.
Applying the same analysis to a chemical process flowsheet surrogate confirmed that these mechanisms transfer to an engineering setting. However, the preferred strategy and sparsity level appears to depend on the case study as well as on the architecture.\\
Overall, our results suggest that pruning is a valuable tool when developing neural network surrogates for optimization tasks, especially in engineering applications requiring repeated optimization solves. In such settings, the additional offline training effort associated with pruning can be compensated by gains in online optimization performance. Future work could consider other pruning criteria, such as sensitivity-based methods, and a deeper analysis of the interaction between pruning and regularization. Furthermore, extending the analysis to smooth activations, which yield NLP rather than MILP formulations, and to reduced-space embeddings would clarify which pruning benefits remain when binary-variable reduction no longer applies.

\section*{Acknowledgments}
We gratefully acknowledge the funding support provided by Shell Global Solutions International B.V.\\
The authors also thank Jing Cui, MSc, for her preliminary work on pruning neural networks for efficient optimization carried out at RWTH Aachen University and Prof. Alexander Mitsos for useful discussions.

\section*{Declaration of generative AI use}
The authors used large language models (including models developed by OpenAI and Anthropic) to assist with editing, grammar refinement, and enhancing the clarity of the manuscript. These tools were also utilized to aid in code generation and debugging. All outputs were carefully reviewed and approved by the authors.

\bibliographystyle{unsrtnat}
\bibliography{references}

\newpage

\FloatBarrier
\appendix

\section*{Appendix}

\section{Detailed proofs for IA-monotonicity under pruning}
\label{app:proofs}

We report the full proofs of the IA-monotonicity results stated in Section~\ref{sec:bound_tightening_pruning}. 
These results formalize the effect of the pruning step itself, i.e., setting a subset of weights to zero while keeping the remaining weights fixed. 
The key observation is that, under interval arithmetic, affine-layer widths propagate through the absolute values of the weights, while ReLU activations cannot increase interval widths. 
As a result, pruning cannot increase IA-based preactivation widths and, under mild non-degeneracy conditions, can lead to a strict reduction that propagates through the network.

\begin{proposition}[Monotonic tightening of big-M bounds under pruning]
\label{prop:monotonic-tightening}
Consider a feedforward ReLU network
$$
    \mathbf{p}^{(j)} = W^{(j)}\mathbf{z}^{(j-1)} + \mathbf{b}^{(j)}, \qquad
    \mathbf{z}^{(j)} = \mathrm{ReLU}\!\left(\mathbf{p}^{(j)}\right), \qquad
    j = 1,\dots,N_L
$$
with input box $\mathbf{x} \equiv \mathbf{z}^{(0)}\in[\mathbf{L}^{(0)},\mathbf{U}^{(0)}]$ and preactivation $\mathbf{p}$.  
Let $(\mathbf{L}^{(j)},\mathbf{U}^{(j)})$ denote the standard interval-arithmetic (IA) preactivation bounds defined as in Eqs.~\eqref{eq:lower_bound-IA}--\eqref{eq:upper_bound-IA}
and let the activation bounds be $\mathbf{z}^{(j)}_L=\max\{\mathbf{0},\mathbf{L}^{(j)}\}$, $\mathbf{z}^{(j)}_U=\max\{\mathbf{0},\mathbf{U}^{(j)}\}$ (componentwise).
Let
$$
    \Delta \mathbf{z}^{(j)} := \mathbf{z}_U^{(j)} - \mathbf{z}_L^{(j)} \in \mathbb{R}_+^{N_N^{(j)}}, \qquad
    \Delta \mathbf{p}^{(j)} := \mathbf{U}^{(j)} - \mathbf{L}^{(j)} \in \mathbb{R}_+^{N_N^{(j)}}
$$
be the corresponding interval widths.

Let $\mathbf{W}'^{(j)}$ be a pruned version of $\mathbf{W}^{(j)}$ obtained by pruning an arbitrary subset of entries, and compute the corresponding IA bounds $(L'^{(j)},U'^{(j)})$ and widths 
$\Delta \mathbf{p}'^{(j)},\Delta \mathbf{z}'^{(j)}$.

Then, for every layer $j$,
$$
\boxed{\;\Delta \mathbf{p}'^{(j)} \le \Delta \mathbf{p}^{(j)}\;}\quad\text{and}\quad
\boxed{\;\Delta \mathbf{z}'^{(j)} \le \Delta \mathbf{z}^{(j)}\;}
\quad\text{(componentwise)}.
$$
Consequently, the big-$M$ gaps $U_i^{(j)}-L_i^{(j)}(=\Delta p^{(j)}_i)$ used to encode
$\mathrm{ReLU}(p^{(j)}_i)$ are never larger after pruning.
\end{proposition}

\begin{proof}
\textbf{Step 1 (Affine layer width is linear in $|W|$).}
For any scalar $w$ and interval $[L,U]$,
$$
\max\{wL,wU\}-\min\{wL,wU\}=|w|\,(U-L).
$$
Applying this term-wise in the IA formulas yields
\begin{equation}\label{eq:affinewidth}
\Delta \mathbf{p}^{(j)} = |W^{(j)}|\,\Delta \mathbf{z}^{(j-1)},
\end{equation}
with $|\cdot|$ the elementwise absolute value. Biases cancel in widths.

\textbf{Step 2 (ReLU is nonexpansive for intervals).}
Since $\mathrm{ReLU}$ is monotone and 1-Lipschitz,
\begin{equation}\label{eq:reluwidth}
\Delta \mathbf{z}^{(j)} \le \Delta \mathbf{p}^{(j)}\
\end{equation}
componentwise (because $(\max\{0,U\}-\max\{0,L\}) \le (U-L)$).

\textbf{Step 3 (Effect of pruning on affine widths).}
Pruning zeroes some entries of $W^{(j)}$; hence $|W'^{(j)}|\le |W^{(j)}|$ elementwise.
Using \eqref{eq:affinewidth} for both networks and the nonnegativity of all widths,
$$
\Delta \mathbf{p}'^{(j)} = |W'^{(j)}|\,\Delta \mathbf{z}'^{(j-1)}
\;\le\; |W^{(j)}|\,\Delta \mathbf{z}'^{(j-1)}
\;\le\; |W^{(j)}|\,\Delta \mathbf{z}^{(j-1)}=\Delta \mathbf{p}^{(j)},
$$
where the second inequality follows by induction on $j$.

\textbf{Step 4 (Propagation through ReLU).}
From Step~3 we have that pruning yields narrower preactivation intervals,
$L'^{(j)} \ge L^{(j)}$ and $U'^{(j)} \le U^{(j)}$.
Since the ReLU function is monotone nondecreasing and
1-Lipschitz, applying it preserves this ordering of interval bounds:
\begin{equation}
\begin{aligned}        
&\mathbf{z}_L'^{(j)} = \text{ReLU}(L'^{(j)}) = \max(0,L'^{(j)}) \ge \max(0,L^{(j)}) = \mathbf{z}_L^{(j)},\\
&\mathbf{z}_U'^{(j)} = \text{ReLU}(U'^{(j)}) = \max(0,U'^{(j)}) \le \max(0,U^{(j)}) = \mathbf{z}_U^{(j)}.
\end{aligned}
\end{equation}

Subtracting the two inequalities gives
\begin{equation}
\Delta \mathbf{z}'^{(j)} = \mathbf{z}_U'^{(j)} - \mathbf{z}_L'^{(j)} \le \mathbf{z}_U^{(j)} - \mathbf{z}_L^{(j)} = \Delta \mathbf{z}^{(j)}.
\end{equation}
Hence the activation widths in the pruned network are no larger than those in the
unpruned network for every layer~$j$, completing the induction.
\end{proof}

The following proposition identifies mild conditions under which pruning yields a strict reduction in the final bound width at an output of interest, i.e., pruning provably improves the continuous relaxation in B\&B.

\begin{proposition}[Strict tightening under mild conditions]
\label{prop:strict-tight}
Under the setting of Proposition~\ref{prop:monotonic-tightening}, fix an output coordinate of interest, denoted by $\mathrm{out}$, and let
$\Delta z_{\mathrm{out}}$ and $\Delta z'_{\mathrm{out}}$ denote the corresponding activation interval widths
in the unpruned and pruned networks, respectively. Assume there exist a layer $j^\star$ and indices $(i,k)$
such that the following conditions hold:
\begin{enumerate}
    \item \textbf{Nontrivial prune:} at least one nonzero weight $W_{i,k}^{(j^\star)}\neq 0$ is pruned to $0$;
    \item \textbf{Upstream signal range:} the input to the pruned weight is not fixed over the domain, i.e.,
    neuron $k$ in layer $j^\star-1$ satisfies $\Delta z_k^{(j^\star-1)}>0$;
    \item \textbf{Downstream influence (active connectivity):} the pruned connection feeds into the selected
    output through at least one intact computation path, i.e., there exists a directed path from neuron $i$
    in layer $j^\star$ to the output coordinate $\mathrm{out}$ using only nonzero weights in the unpruned network;
    \item \textbf{No dead/degenerate ReLUs on that path:} along that path, no neuron is provably always off
    (its preactivation interval lies entirely in $(-\infty,0]$) and no neuron has zero preactivation width.
\end{enumerate}
Then the interval width of the chosen output coordinate is strictly smaller after pruning:
\[
\boxed{\;\Delta z'_{\mathrm{out}} < \Delta z_{\mathrm{out}}\;}
\]
\end{proposition}

\begin{proof}
By Proposition~\ref{prop:monotonic-tightening}, pruning cannot increase interval widths anywhere in the network. In particular,
\[
\Delta z_i'^{(j)} \le \Delta z_i^{(j)}
\qquad
\text{for all layers } j=1,\dots,N_L \text{ and neurons } i=1,\dots,N_N^{(j)}.
\]

At layer $j^\star$, the affine-width identity \eqref{eq:affinewidth} gives
\[
\Delta p_i^{(j^\star)}
=
\sum_{m=1}^{N_N^{(j^\star-1)}} |W_{i,m}^{(j^\star)}|\,\Delta z_m^{(j^\star-1)},
\]
and analogously for the pruned network,
\[
\Delta p_i'^{(j^\star)}
=
\sum_{m=1}^{N_N^{(j^\star-1)}} |W_{i,m}'^{(j^\star)}|\,\Delta z_m'^{(j^\star-1)}.
\]
Since pruning sets $W_{i,k}^{(j^\star)}$ to zero and does not increase any upstream widths,
\[
\Delta p_i'^{(j^\star)}
\le
\sum_{m\ne k} |W_{i,m}^{(j^\star)}|\,\Delta z_m^{(j^\star-1)}
=
\Delta p_i^{(j^\star)} - |W_{i,k}^{(j^\star)}|\,\Delta z_k^{(j^\star-1)}.
\]
By Conditions~1 and~2, the subtracted term is strictly positive, and therefore
\[
\Delta p_i'^{(j^\star)} < \Delta p_i^{(j^\star)}.
\]

By Condition~4, the ReLU at neuron $i$ is not degenerate on the domain, so this strict decrease in
preactivation width induces a strict decrease in activation width:
\[
\Delta z_i'^{(j^\star)} < \Delta z_i^{(j^\star)}.
\]

By Condition~3, neuron $i$ influences the selected output coordinate through at least one intact downstream path.
Along that path, all affine transformations depend monotonically on incoming widths through nonnegative coefficients,
and Condition~4 ensures that no intermediate ReLU collapses the transmitted variation. Since all other widths are
nonincreasing under pruning, the strict decrease at neuron $(j^\star,i)$ propagates to the chosen output coordinate.
Hence,
\[
\Delta z'_{\mathrm{out}} < \Delta z_{\mathrm{out}}.
\]
\end{proof}
\section{Model parameters}
This appendix summarizes the dataset characteristics and training hyperparameters used for the neural network models in the case studies considered in this work.
\FloatBarrier
\label{app:trainig_parameters}
\subsection{Peaks function case study}

\begin{table}[htb!]
\centering
\caption{Dataset parameters used in our experiments for the peaks function case study.}
\begin{tabular}{ll}
\toprule
Parameter & Value \\
\midrule
Training + validation samples & 5{,}000 \\
Training/validation split & 0.8 / 0.2 \\
Test data samples & 500 \\
\bottomrule
\end{tabular}
\label{tab:peaks_training_dataset}
\end{table}

\begin{table}[htb!]
\centering
\caption{Training hyperparameters used in our experiments for the peaks function case study.}
\begin{tabular}{ll}
\toprule
Hyperparameter & Value \\
\midrule
Epochs & 50{,}000 \\
Learning rate & $10^{-5}$ \\
Optimizer & Adam \\
Batch size & 1{,}000 \\
$L_2$ regularization factor & $5 \times 10^{-7}$ \\
MAPE early stopping tolerance & $0.001\%$ \\
Early stopping patience & 15 iterations \\
\bottomrule
\end{tabular}
\label{tab:peaks_training_hyperparameters}
\end{table}

\subsection{Flowsheet optimization case study}
\begin{table}[htb!]
\centering
\caption{Inputs of the neural network surrogate for the autothermal reformer flowsheet.}
\begin{tabular}{lc}
\toprule
Variable & Units \\
\midrule
Bypass fraction ($\alpha$) & - \\
Natural gas to steam ratio ($r$) & - \\
\bottomrule
\end{tabular}
\label{tab:nn_surrogate_inputs}
\end{table}

\begin{table}[htb!]
\centering
\caption{Outputs of the neural network surrogate for the autothermal reformer flowsheet.}
\begin{tabular}{lc}
\toprule
Variable & Units \\
\midrule
Steam flow rate ($F_{\text{steam}}$) & mol/s \\
Reformer heat duty ($Q_{\text{ref}}$) & MW \\
$x_{\text{Ar}}$ & - \\
$x_{\text{C}_2\text{H}_6}$ & - \\
$x_{\text{C}_3\text{H}_8}$ & - \\
$x_{\text{C}_4\text{H}_{10}}$ & - \\
$x_{\text{CH}_4}$ & - \\
$x_{\text{CO}}$ & - \\
$x_{\text{CO}_2}$ & - \\
$x_{\text{H}_2}$ & - \\
$x_{\text{H}_2\text{O}}$ & - \\
$x_{\text{N}_2}$ & - \\
$x_{\text{O}_2}$ & - \\
\bottomrule
\end{tabular}
\label{tab:nn_surrogate_outputs}
\end{table}

\begin{table}[htb!]
\centering
\caption{Dataset parameters used for the flowsheet optimization case study.}
\begin{tabular}{ll}
\toprule
Parameter & Value \\
\midrule
Training + validation samples & 2{,}500 \\
Training/validation split & 0.8 / 0.2 \\
Test data samples & 300 \\
\bottomrule
\end{tabular}
\label{tab:reformer_training_dataset}
\end{table}

\begin{table}[htb!]
\centering
\caption{Training hyperparameters used in our experiments for the flowsheet optimization case study.}
\begin{tabular}{ll}
\toprule
Hyperparameter & Value \\
\midrule
Epochs & 60{,}000 \\
Learning rate & $10^{-4}$ \\
Optimizer & Adam \\
Batch size & 500 \\
$L_2$ regularization factor & $5 \times 10^{-7}$ \\
MAPE early stopping tolerance & $0.001\%$ \\
Early stopping patience & 20 iterations \\
\bottomrule
\end{tabular}
\label{tab:reformer_training_hyperparameters}
\end{table}

\FloatBarrier

\section{Additional experiments and results}
This appendix reports supporting experiments referenced in the main text.
\subsection{Preliminary analysis on regularization on the chemical process flowsheet}
\label{app:regularization_analysis}
We adopted $L_2$ with $\lambda = 5\times10^{-7}$ for the main study because it perturbs the accuracy of trained network least. This choice is supported also on preliminary analysis on the engineering case study, summarized in Table~\ref{tab:reg_strength_reformer} below.

\begin{table}[htbp]
  \centering
  \caption{Accuracy--speedup trade-off of each regularization setting on the reformer case study. $\Delta$MAPE is the MAPE variation against the unpruned baseline of the same architecture; speedup is the ratio of the baseline's optimization time to the pruned model's. Both are medians over every pruned configuration of the setting.}
  \label{tab:reg_strength_reformer}
  \begin{tabular}{lrrrr}
    \toprule
    Regularizer & \multicolumn{2}{c}{WP} & \multicolumn{2}{c}{NP} \\
    \cmidrule(lr){2-3} \cmidrule(lr){4-5}
     & $\Delta$MAPE (\%) & Speedup & $\Delta$MAPE (\%) & Speedup \\
    \midrule
    None & 0.04 & $1.3\times 10^{4}$ & 0.09 & $7.6\times 10^{3}$ \\
    \midrule
    $L_1$ $5\times 10^{-7}$ & 0.03 & $1.0\times 10^{4}$ & 0.06 & $1.0\times 10^{4}$ \\
    $L_1$ $10^{-4}$ & 0.89 & {\boldmath$3.4\times 10^{4}$} & 0.18 & {\boldmath$3.0\times 10^{4}$} \\
    $L_1$ $10^{-2}$ & 13.81 & $2.9\times 10^{3}$ & 12.11 & {\boldmath$3.0\times 10^{4}$} \\
    \midrule
    $L_2$ $5\times 10^{-7}$ & \textbf{-0.01} & $8.8\times 10^{3}$ & \textbf{0.05} & $6.6\times 10^{3}$ \\
    $L_2$ $10^{-4}$ & 0.05 & $1.7\times 10^{4}$ & \textbf{0.05} & $1.1\times 10^{4}$ \\
    $L_2$ $10^{-2}$ & 4.49 & $2.1\times 10^{4}$ & 4.47 & $1.6\times 10^{4}$ \\
    \bottomrule
  \end{tabular}
\end{table}

\subsection{Accuracy-runtime trade-off with HiGHS}
\label{app:highs_pareto}

Figure~\ref{fig:speed_mape_highs} reproduces the Pareto-front analysis of
Section~\ref{sec:exp-comp-perf} with the open-source solver HiGHS instead of
Gurobi. The overall pattern is consistent with Gurobi across complexity levels
and morphologies. No morphology or complexity level shows a different preferred technique under HiGHS than under Gurobi.
The two solvers diverge mainly in absolute runtime at the largest
complexities. Optimization times under HiGHS grow relative to Gurobi as
network size increases, and by 200k parameters several configurations that
converge well within the time limit under Gurobi approach it under HiGHS.
This can likely be explained by the two solvers implementing different
presolve and branching heuristics, whose relative efficiency is not preserved
as problem size grows.

\begin{figure}
    \centering
    \includegraphics[width=0.9\linewidth]{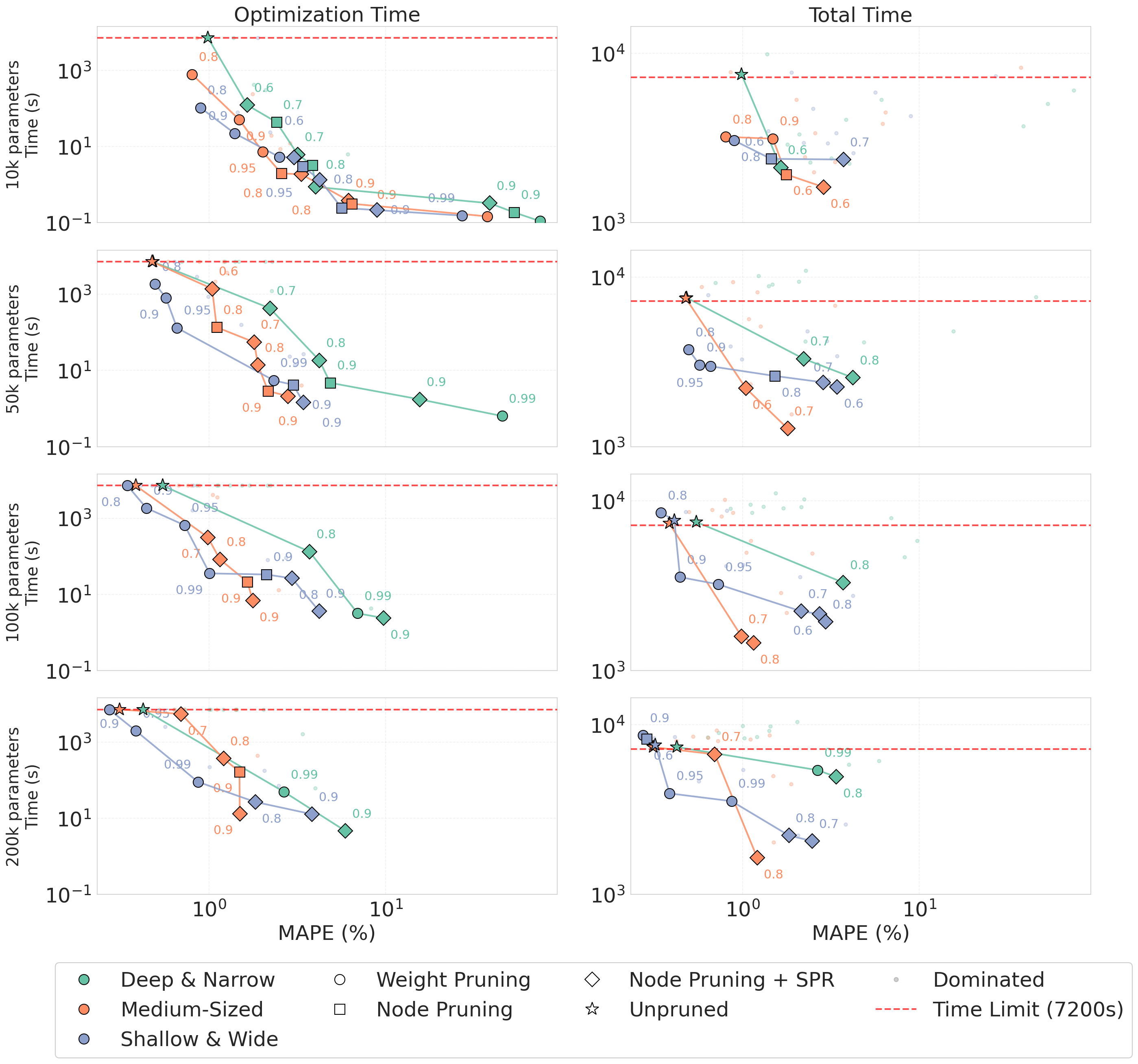}
    \caption{Accuracy-runtime Pareto fronts across network complexity and
morphology, for optimization time (left) and total time including training
and pruning (right), solved with HiGHS.}
    \label{fig:speed_mape_highs}
\end{figure}

\subsection{Masked-dense vs. reduced formulation of pruned networks}
\label{app:formulation_equivalence}
A node-pruned network can be written out as a smaller dense network by deleting the rows and columns of the vanished units from the weight matrices. For the medium-sized architecture ($124 \times 4$, 50k parameters) of the peaks case study we solved every node-pruned model twice: once as the masked-dense formulation built through OMLT and left to presolve, and once as the directly reduced network parsed into Pyomo. As built, the two differ substantially in size (2528 against 1022 rows at $60\%$ sparsity), but declare the same number of binary variables. This is because a masked unit has a degenerate pre-activation interval, so its state is fixed and OMLT creates no binary for it. Hence, the excess in the masked-dense encoding is entirely continuous variables and linear constraints. Gurobi's presolve then reduces both encodings to the same MILP (Table~\ref{tab:formulation_equivalence}), identical in rows, columns, nonzeros, binaries and root LP bound in $n/n$ paired runs. The reduced formulation is itself reduced further, by ${\approx}26\%$ of rows but under $2\%$ of binaries at every sparsity level. With the relaxation identical, any residual difference can only arise in the search. Figure~\ref{fig:reduced_masked} plots each paired solve against parity for solve time, nodes explored, simplex throughput and work units. Since no component of the reported speedup exists either only in the masked-dense or only in the reduced route, we conclude that it is intrinsic to the pruned model.

\begin{table}[htbp]
  \centering
  \caption{Masked-dense (MD) versus natively reduced (R) formulation for the medium-sized architecture (126 $\times$ 4, 50k params.), before and after Gurobi's presolve. The masked-dense model keeps the full network width with pruned weights fixed to zero; the reduced model is the smaller dense network parsed directly into Pyomo.}
  \label{tab:formulation_equivalence}
  \begin{tabular}{llcccc}
    \toprule
    \multicolumn{2}{c}{} & \multicolumn{4}{c}{Final sparsity (node pruning)} \\
    \cmidrule(lr){3-6}
    Quantity & Formulation & 0.6 & 0.7 & 0.8 & 0.9 \\
    \midrule
    Model as built & \# Rows - (MD) & $2528$ & $2528$ & $2528$ & $2528$ \\
     & \# Rows - (R) & $1022$ & $762$ & $502$ & $242$ \\
     & \# Columns - (MD) & $1222$ & $1170$ & $1118$ & $1066$ \\
     & \# Columns - (R) & $616$ & $460$ & $304$ & $148$ \\
     & \# Int. Vars - (MD) & $204$ & $152$ & $100$ & $48$ \\
     & \# Int. Vars - (R) & $204$ & $152$ & $100$ & $48$ \\
    \midrule
    After presolve & \# Rows - (MD) & $757 \pm 2$ & $564 \pm 4$ & $374 \pm 2$ & $180 \pm 1$ \\
     & \# Rows - (R) & $757 \pm 2$ & $564 \pm 4$ & $374 \pm 2$ & $180 \pm 1$ \\
     & \# Columns - (MD) & $559 \pm 1$ & $417 \pm 2$ & $276 \pm 1$ & $134$ \\
     & \# Columns - (R) & $559 \pm 1$ & $417 \pm 2$ & $276 \pm 1$ & $134$ \\
     & \# Int. Vars - (MD) & $200 \pm 1$ & $149 \pm 2$ & $99 \pm 1$ & $48$ \\
     & \# Int. Vars - (R) & $200 \pm 1$ & $149 \pm 2$ & $99 \pm 1$ & $48$ \\
    \midrule
    Relaxation & Root LP bound - (MD) & $-140 \pm 29$ & $-79 \pm 26$ & $-24 \pm 12$ & $-14 \pm 12$ \\
     & Root LP bound - (R) & $-140 \pm 29$ & $-79 \pm 26$ & $-24 \pm 12$ & $-14.37 \pm 12$ \\
     & Presolve time (ms) - (MD) & $64 \pm 9$ & $30$ & $10$ & $0$ \\
     & Presolve time (ms) - (R) & $78 \pm 4$ & $38 \pm 8$ & $18 \pm 4$ & $0$ \\
    \bottomrule
  \end{tabular}
\end{table}

\begin{figure}
    \centering
    \includegraphics[width=\linewidth]{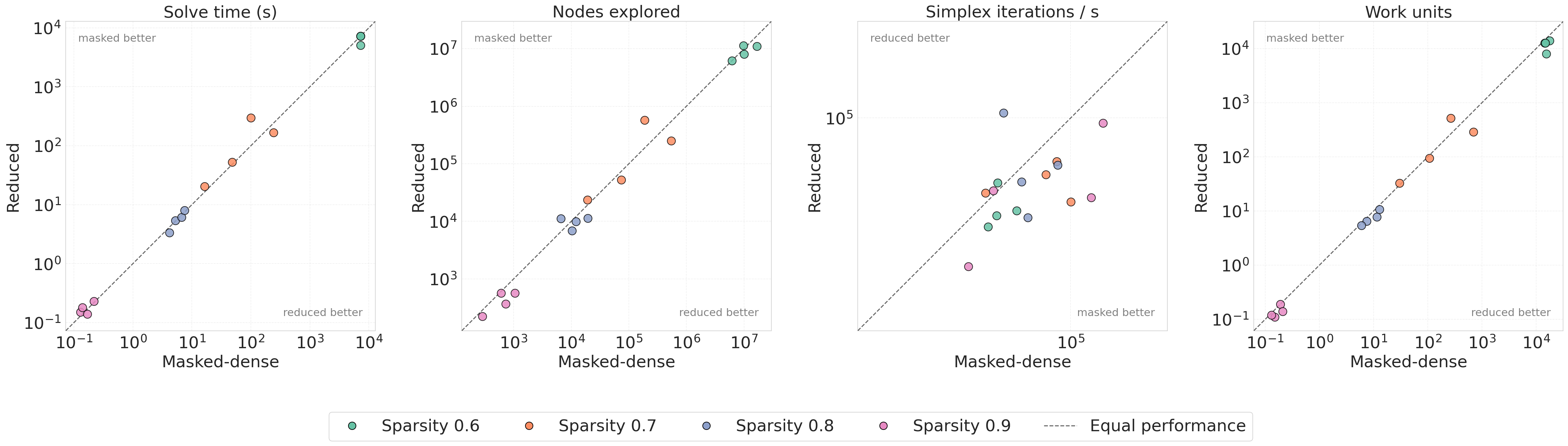}
    \caption{Masked-dense versus natively reduced formulation, one marker per run solved both ways ($126 \times 4$ architecture, node pruning). Points scatter about parity in every panel,  which demonstrates a no systematic advantage for either formulation.}
    \label{fig:reduced_masked}
\end{figure}

\subsection{Layer-wise bound tightening in deep networks}
\label{app:bounds-deep}
Figure~\ref{fig:bounds_deep_NN} provides a detailed layer-wise analysis of the bound tightening effect for the deep and narrow architecture (102$\times$20) under increasing sparsity levels.
In contrast to the averaged results shown in the main text, this figure reports the distribution of big-M widths $U_i^{(j)} - L_i^{(j)}$ across individual layers and neurons.
For the unpruned network, bound widths grow exponentially with depth, showing a strong amplification effect in the deeper layers due to IA; pruning suppresses this explosion and stabilizes the bounds throughout the network.
Under weight pruning, the reduction is particularly pronounced in the middle layers, suggesting that a larger fraction of parameters is effectively eliminated there.
This pattern indicates that the network tends to concentrate information in a lower-dimensional intermediate representation, reminiscent of an encoder--decoder structure.
Such a layer-specific effect is less evident with node pruning, since the same sparsity level is imposed uniformly across layers by design.
\begin{figure}[h!]
    \centering
    \begin{subfigure}{0.48\linewidth}
        \centering
        \includegraphics[width=\linewidth]{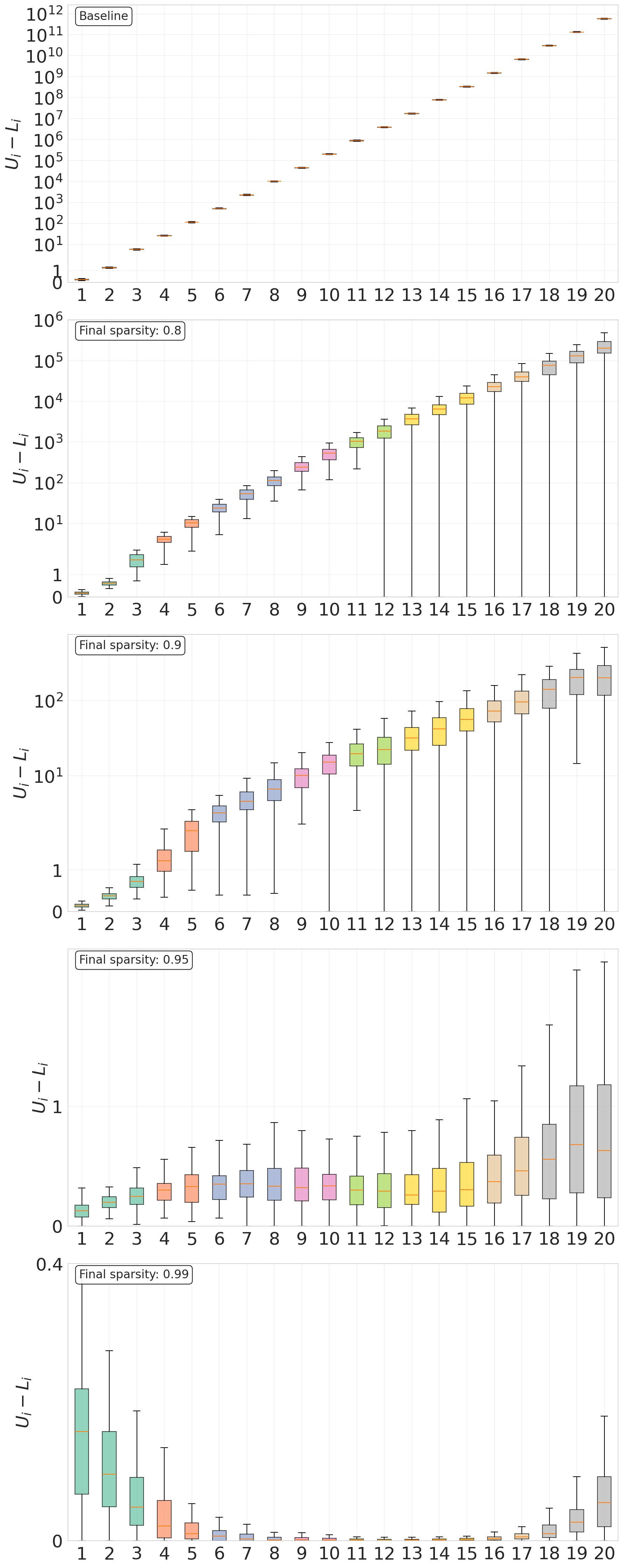}
        \caption{Weight pruning}
        \label{fig:bounds_deep_NN_weight}
    \end{subfigure}
    \hfill
    \begin{subfigure}{0.48\linewidth}
        \centering
        \includegraphics[width=\linewidth]{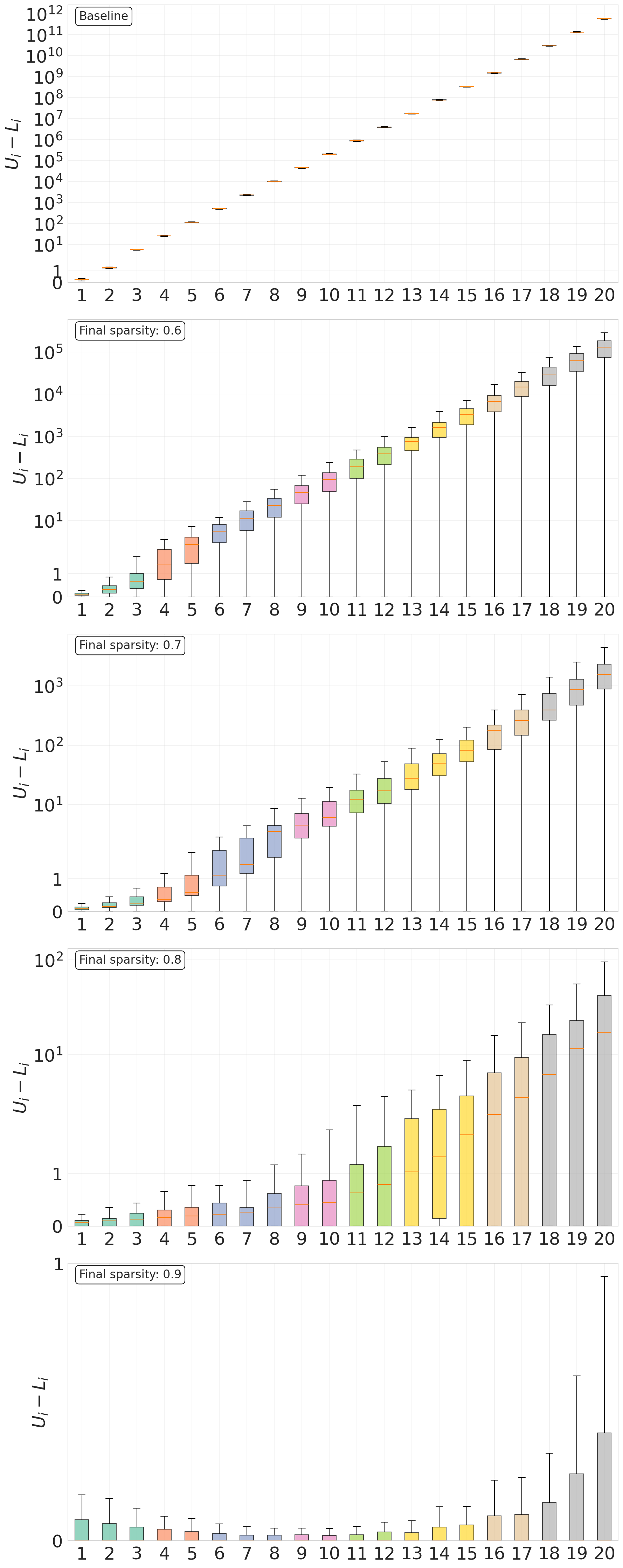}
        \caption{Node pruning}
        \label{fig:bounds_deep_NN_node}
    \end{subfigure}
    \caption{Layer-wise big-M bound widths under increasing sparsity for the deep and narrow architecture (102$\times$20).
    The unpruned network exhibits exponential growth of bound widths with depth due to IA amplification, whereas pruning reduces bounds throughout the network.
    Weight pruning induces particularly strong reductions in the middle layers, suggesting concentration of information in a lower-dimensional intermediate representation (encoder--decoder-like behavior).
    This layer-specific effect less evident under node pruning, which enforces uniform sparsity across layers by design.}
    \label{fig:bounds_deep_NN}
\end{figure}

\end{document}